\documentclass[12pt,a4paper,reqno]%
{amsart}
\usepackage{amsfonts, amsmath, amssymb, amsthm, amsbsy, amscd, euscript}

\usepackage[english]{babel}

\usepackage[%
unicode=true,
colorlinks=false
]{hyperref}

\newtheorem{prop}{Proposition}
\newtheorem{cor}[prop]{Corollary}
\newtheorem{lemma}[prop]{Lemma}

\theoremstyle{definition}
\newtheorem{example}{Example}
\newtheorem{define}{Definition}
\newtheorem*{notation}{Notation}
\newtheorem{claim}[prop]{Claim}

\theoremstyle{remark}
\newtheorem{rem}{Remark}

\newcommand{\pinner}{\mathbin{\mathchoice
   {\hbox{\vrule width0.6em depth0pt height0.4pt
   \vrule width0.4pt depth0pt height0.8ex}}
   {\hbox{\vrule width0.6em depth0pt height0.4pt
   \vrule width0.4pt depth0pt height0.8ex}}
   {\hbox{\kern0.14em
   \vrule width0.48em depth0pt height0.4pt
   \vrule width0.4pt depth0pt height0.6ex\kern0.14em}}
   {\hbox{\kern0.1em
   \vrule width0.39em depth0pt height0.4pt
   \vrule width0.4pt depth0pt height0.5ex\kern0.1em}}}}
\newcommand{\inner}{\pinner\,}

\newcommand{\cD}{\mathcal{D}}

\DeclareMathOperator{\Mat}{Mat}
\DeclareMathOperator{\Ann}{Ann}

\DeclareMathOperator{\Lie}{L}
\DeclareMathOperator{\Vect}{Vect}

\DeclareMathOperator{\dord}{dord}

\newcommand{\gothg}{\mathfrak{g}}

\newcommand{\cA}{{\EuScript A}}
\newcommand{\cB}{{\EuScript B}}
\newcommand{\cE}{\mathcal{E}}
\newcommand{\cEinf}{\mathcal{E}^{\infty}}

\newcommand{\cC}{\mathcal{C}}
\newcommand{\veps}{\varepsilon}
\newcommand{\BBR}{\mathbb{R}}

\newcommand{\BBC}{\mathbb{C}}
\newcommand{\BBZ}{\mathbb{Z}}
\newcommand{\bu}{\boldsymbol{u}}
\newcommand{\bv}{{\boldsymbol{v}}}
\newcommand{\bw}{{\boldsymbol{w}}}
\newcommand{\dd}{\partial}
\newcommand{\ddb}{\boldsymbol{\partial}}
\newcommand{\Id}{{\mathrm d}}
\newcommand{\fg}{\mathfrak{g}}

\newcommand{\gl}{\mathfrak{gl}}
\newcommand{\gsl}{\mathfrak{sl}}

\newcommand{\fgl}{\mathfrak{sl}}

\newcommand{\bbl}{\underline{\boldsymbol{[}}}
\newcommand{\bbr}{\underline{\boldsymbol{]}}}

\newcommand{\fnl}{[}
\newcommand{\fnr}{]^{\text{FN}}}

\newcommand{\bone}{{\bar{1}}}
\newcommand{\bzero}{{\bar{0}}}

\newcommand{\parity}{\mathsf{p}}

\newcommand{\ii}{\boldsymbol{i}}

\newcommand{\ao}{\alpha^{\bone}}
\newcommand{\az}{\alpha^{\bzero}}
\newcommand{\bo}{\beta^{\bone}}
\newcommand{\bz}{\beta^{\bzero}}

\title[On the (non)\/removability of spectral parameters in
$\mathbb{Z}_2$-\/graded~ZCR]
{On the (non)\/removability of spectral parameters\\ 
in $\mathbb{Z}_2$-\/graded zero\/-\/curvature representations\\
and its applications}

\author[A.~V.~Kiselev]{Arthemy V.~Kiselev${}^{*}$}
\thanks{${}^{*}$%
  \textit{Address}:
  Bernoulli Institute for Mathematics, Computer Science and Artificial Intelligence,
  University of Groningen,
  P.O.Box~407, 9700\,AK Groningen, The Netherlands.\\%\quad
  \textit{E-mail}: \texttt{A.V.Kiselev\symbol{"40}rug.nl}}
\thanks{
  \textit{Current address}:
  Institut des Hautes \'Etudes Scientifiques (IH\'ES),
  Le Bois--Marie 35, route de Chartres, Bures-sur-Yvette, F-91440 France}

\author[A.~O.~Krutov]{Andrey O.~Krutov$^{\dag}$}  %
\thanks{${}^{\dag}$%
  Author to whom correspondence should be addressed.\quad  
  \textit{Address}: %
  Institute of Mathematics,
  Polish Academy of Science,
  ul. \'{S}niadeckich 8, 00-656 Warsaw, Poland.
\quad \textit{E-mail}:
 \texttt{a.krutov\symbol{"40}impan.pl}}
\thanks{\textit{Present address}: %
Independent University of Moscow, Bol.\ %
Vlasyevskiy per.~11, Moscow 119002 Russia.%
}

\date{21 November 2016, revised 15 April 2018, in final form 15 June 2018}



\begin{document}
\begin{abstract}
We generalise to the $\mathbb{Z}_2$-\/graded set-\/up a practical %
method for inspecting the (non)\/removability of %
parameters in zero-\/curvature representations for partial differential equations (PDEs) under the action of smooth families of gauge transformations.
We illustrate the generation and elimination of parameters in %
the flat structures over $\mathbb{Z}_2$-\/graded PDEs
by analysing the link between deformation of zero-\/curvature representations %
via infinitesimal gauge transformations and, on the other hand, 
propagation of linear coverings over PDEs %
using the Fr\"olicher--\/Nijenhuis
bracket.
\end{abstract} 
\maketitle
\subsection*{Introduction}
The aim of the present paper is to provide a tool for inspection of (non)re\-mo\-va\-bi\-li\-ty of
parameters in matrix Lie super-\/algebra valued zero-\/curvature representation for $\BBZ_2$-\/graded
partial differential equations, e.g., for supersymmetric equations of Korteweg--\/de Vries (KdV) type.  This
work concludes the cycle of papers~\cite{MathieuNew,HKKW,JMP2012} in which the Gardner deformation
problem~\cite{Miura68,Gardner} for the triplet $a\in\{-2,1,4\}$ of P.~Mathieu's $N=2$ supersymmetric KdV
equations~\cite{MathieuNew,MathieuTwo} was addressed by using the various geometric, analytic, and computational
techniques. In the paper~\cite{HKKW} (%
joint with Hussin and Wolf) 
we presented our first solution of the Open Problem~2
from~\cite{MathieuOpen}, constructing the hierarchy of Hamiltonian super\/-\/functionals for the $N=2$
super-\/KdV with~$a=4$ on the basis of a classical Gardner's deformation for the Kaup\/--\/Boussinesq
hierarchy of bosonic limit of that super-\/system. Equally well applicable to many other supersymmetric or
just $\BBZ_2$-\/graded completely integrable systems, this line of reasoning was furthered
in~\cite{Cyprus2014}.  The alternative solution of Gardner's deformation problem, which we proposed
in~\cite{JMP2012}, see also~\cite[\S~5.3]{AOKThesis}, allows to address that problem from a different perspective, namely, by putting
it in the context of finding Lie
super\/-\/algebra valued zero\/-\/curvature representations for PDEs or, in broader terms, finding the
parametric families of non\/-\/local structures over such systems
(see~\cite{WahlquistEstabrook75,WahlquistEstabrook76}, also~\cite{KuperIrish,Fordy83,TMPh2006}). The key idea
is that the \emph{deformation} parameter, the presence of which yields %
an explicit recurrence relations
between the hierarchy of integrals of motion for the PDE under study, is an avatar of the \emph{spectral}
parameter in the inverse scattering problem~\cite{Faddeev,ZSh}. This correspondence, itself referring to the
task of switching between the different realisations of Lie
algebras~\cite{Olver1992,PopovychBoykoNesterenko2003,Shchepochkina2006,RoelofsThesis,Nesterenko2014}, necessitates the inspection of
relevance vs removability of the parameters in %
families of non\/-\/local structures.

We resolve this issue of (non)\/removability by approaching the problem from two different
directions. First, we formulate a ready\/-\/to\/-\/use procedure for inspecting the (non)\/removability of
parameters in Lie super\/-\/algebra valued zero\/-\/curvature representations for~PDEs under the action of
smooth families of gauge transformations.\footnote{Zero\/-\/curvature representations %
for partial differential equations %
are the input data for realisation of the inverse scattering method.
For PDEs with unknown functions in two independent variables, %
the most interesting zero\/-\/curvature representations are those which contain a non\/-\/removable spectral parameter; 
in that case the system of~PDEs %
can be kinematically integrable~\cite{Faddeev,ZSh}.}
This practical algorithm is the $\BBZ_2$-\/graded generalisation of
Marvan's technique~\cite{Marvan2002,Marvan2010}, which was designed initially for purely bosonic systems;
a technique for solution of the same problem was developed independently by Sakovich in~\cite{Sakovich95,Sakovich2004}.  
Extending %
Marvan's %
approach %
in section~\ref{GradedMarvan} %
to, %
e.g., %
the geometry of Mathieu's $N=2$ supersymmetric Korteweg\/--\/de Vries equations~\cite{MathieuNew},
we now prove for example that one cannot remove in that way the %
parameter which is contained in the structure found by Das \textit{et al.} in~\cite{Das} and which was
used later in~\cite{JMP2012} within our %
alternative solution to Mathieu's problem of deforming his $N=2$, $a=4$ super\/-\/KdV equation.

Another powerful technique for generating the parametric families of non-\/local structures over PDEs is
based on the use of Fr\"olicher\/--\/Nijenhuis bracket.  This approach was elaborated %
by Krasil'shchik \textit{et al.}  in~\cite{KKIgonin2003,JKIgonin2002,JKFlatCon}, cf.~\cite{AVKVestnikMGU2002};
basic facts and ideas from this formalism are recalled in section~\ref{GradedFN} (see also
sec.~\ref{Notation}).
In this geometric context, we explore the nature %
and properties of the parameters found by Gardner~\cite{Miura68} for the
classical Korteweg--\/de Vries equation 
(see Example~\ref{exThiune} on p.\,\pageref{exThiune})  
and by Das \textit{et al.}~\cite{Das%
} for the $N{=}2$ supersymmetric $a{=}4$
Korteweg\/--\/de Vries equation (see %
p.\,\pageref{exThiuneN2}).
Furthermore, using the results described in section~\ref{GradedFN} and techniques from~\cite{BaranMarvan2006},
we prove the integrability of new fifth order $N{=}1$ supersymmetric evolution equation~\eqref{eq5ordF} found by Tian and Liu
(see~\cite{TianLiu5ord,TianWang2016}): 
in Example~\ref{Ex98} on p.~\pageref{Ex98} we construct its~$\gsl(9|8)$-valued zero\/-\/curvature representation with nonremovable parameter.

The main instrument of our study is the notion of prolongation structures over partial differential equations; this concept goes back to Wahlquist and Estabrook~\cite{WahlquistEstabrook75,WahlquistEstabrook76,DoddFordy83,DoddFordy84,Kaup80}.
Let $x^1$,\ $\ldots$,\ $x^n$ be the independent variables in a given PDE~$\cE$ and ${\bar{D}}_{x^i}$~be the restrictions of the respective total derivatives 
to (the infinite prolongation~$\cE^\infty$ of) the equation~$\cE$.
The prolongation structure over~$\cE^\infty$ is described by the Maurer\/--\/Cartan equation that holds restricted to~$\cE^\infty$ (which is denoted by~$\doteq$), %
\[
[\bar{D}_{x^i}+A_i, \bar{D}_{x^j}+A_j]\doteq0, \qquad 1\leqslant i<j\leqslant n. 
\]
In the case of zero\/-\/curvature representations, the objects %
$A_i$ and $A_j$ are Lie (super-) %
algebra valued
functions on (the infinite prolongation~$\cE^\infty$ of) the equation~$\cE$.
In the case of linear coverings, the objects $A_i$ and $A_j$ are vertical vector fields along the fibres~$W$ in a (vector) bundle over~$\cE^\infty$.
This establishes the correspondence between zero\/-\/curvature representations and linear coverings. Indeed, each zero\/-\/curvature
representation with coefficients belonging to a %
Lie (super-)\/algebra 
determines a linear covering, whereas each covering with %
fibre~$W$ can be regarded as a zero\/-\/curvature representation 
the coefficients of which take values in the Lie algebra of vector fields 
on~$W$.
This correspondence %
very often allows one to transfer the
results from one geometry to the other. 
Lemma~\ref{lemmaCommute} and
Proposition~\ref{propFNMarvan} in section~\ref{GradedFN} %
confirm this general principle; similar results
were considered in~\cite{KKIgonin2003}.

Finally, we analyse the link between the two deformation methods in the case of
$\BBZ_2$-\/graded~PDEs. %
In particular, in section~\ref{secZCRCovering}
we illustrate that link %
by switching between the realisations of Lie super\/-\/algebras in zero\/-\/curvature representations for Mathieu's $N=2$, $a=4$ super\/-\/KdV and other integrable systems of Korteweg\/--\/de Vries type.\\[1pt]
\centerline{\rule{1in}{0.7pt}}

\noindent%
This paper is structured as follows. 
First, in section~\ref{Notation} we recall basic facts from the local and non\/-\/local geometry of $\BBZ_2$-\/graded
partial differential equations and we fix some notation.
In section~\ref{GradedMarvan} we generalise --\,to the case of
$\BBZ_2$-graded partial differential equations\,-- Marvan's approach to proving the (non)\/re\/mo\/va\/bi\/li\/ty of
parameters in zero\/-\/curvature representations. 
In section~\ref{GradedFN} we analyse the construction of families of coverings by using the Fr\"o\-li\-cher\/--\/Nijenhuis bracket.
In section~\ref{secZCRCovering} we explore a relation between
the two ways to obtain parametric families of those geometric structures, namely, 
zero\/-\/curvature representations and %
coverings. 
In conclusion on p.~\pageref{pConclusion} 
we summarise the result and list open problems. 
Two %
appendices supersede the main text; in particular, %
a technical proof is %
contained in Appendix~%
\ref{AppClaimZCRDas}.
In Appendix~\ref{appParamNotes}, which is found at
the end of this paper, we discuss the ``unconventional'' strategies for elimination of parameters~--- e.g.,
those which cannot be removed by using smooth families of gauge transformations.  For example, we then %
re\/-\/consider Sasaki's result~\cite{Sasaki79} %
in the context of%
~\cite{KKIgonin2003,JKIgonin2002,JKFlatCon}.

We refer to the books~\cite{Berezin,QFS} for definitions and basic concepts from supergeometry. Let us emphasise that in parallel with supergeometry,
many objects in the geometry of PDEs can be approached by viewing them as the spectra of given rings of functions (that is, as the ringed spaces), see~\cite{JKLStarCov,JKVerb2011}. %
Not only is the definition of supermanifold natural in this context 
(indeed, it is a topological space equipped with a sheaf of smooth
functions taking values in some Grassmann algebra) but also the construction of infinite prolongations~$\cE^\infty$ for partial differential equations~$\cE$ becomes handy; for unlike the $k$th order prolongations~$\cE^{(k)}$ the dimension of which stays finite, the loci~$\cE^\infty$ are usually infinite\/-\/dimensional (see~\cite{JKLStarCov,JKVerb2011,VinogradovCSpecI,VinogradovCSpecII,BVV}). 
Therefore, it is this algebraic formalism (cf.~\cite{Nestruev})
which now ensures %
an almost automatic extension of practical techniques within the formal geometry of PDEs to the $\BBZ_2$-\/graded, supergeometric set\/-\/up.

\section{Preliminaries: %
$\BBZ_2$-\/graded infinite jet bundles}\label{Notation}
\noindent%
In this section we recall necessary definitions from supergeometry
(we refer to~\cite{Berezin,QFS,Leites80} and~\cite{BVV,KK2000,GDE2012,Olver} 
for further details); this material is standard.

\subsection{Jet spaces}%
Let $M^n$ be an $n$-\/dimensional oriented smooth real manifold.  Let us consider two vector
bundles over the same base $M^n$, namely, $\pi^0 \colon E^{m_0+n} \to M^n$ and
$\pi^1 \colon \varSigma^{m_1+n} \to M^n$ with fibre dimensions $m_0$ and $m_1$, respectively. %
(In particular, we let $n=2$ so that the independent variables are $x^1=x$ and $x^2=t$; we have that $m_0=1$,
$m_1=0$ for the Korteweg\/--\/de Vries equation,
$m_0=2$, $m_1=0$ for the hierarchy of the Kaup\/--\/Boussinesq equation, and $m_0=2$, $m_1=2$ for the $N{=}2$
supersymmetric KdV equation, see~\cite{MathieuNew,MathieuOpen,Berezin,QFS}.)

Let $\pi^{\bone} = \Pi\pi^1$ be the odd neighbour of the vector bundle $\pi^1$ ($\Pi$ denotes the reversion of parity).  
By definition, this neighbour is the vector bundle $\pi^{\bone}\colon\Pi\varSigma^{n+m_1}\to M^n$
over the same base and with the same vector space $\BBR^{m_1}$ take as prototype for the fibres.  The
coordinates $\xi^1, \ldots, \xi^{m_1}$ along the fibres $(\pi^{\bone})^{-1}(x)\simeq \BBR^{m_1}$ over $x\in M^n$
are proclaimed $\BBZ_2$-\/parity odd and coordinates $u^1, \ldots, u^{m_0}$ along the fibre
$(\pi^{\bzero})^{-1}(x)\simeq \BBR^{m_0}$ are proclaimed $\BBZ_2$-\/parity even, i.e.\ we introduce the
$\BBZ_2$-\/grading \label{pParity}%
$\parity\colon x^i\mapsto \bzero$,\ %
${\xi^k\mapsto\bone}$,\ $u^j\mapsto\bzero$ for the generators of the ring of smooth $\BBR$-\/valued functions on
the total space $\Pi\varSigma^{m_1+n}$ of the superbundle.
We have that $C^\infty(\Pi\varSigma^{m_1+n}) \simeq
\Gamma(\bigwedge^{\bullet}(\varSigma^{m_1+n})^*)$, where $(\varSigma^{m_1+n})^*$ denotes the space of fibrewise\/-\/linear functions on $\varSigma^{m_1+n}$. 
Finally, let us construct the Whitney sum  $\pi = \pi^{\bzero}\times_{M^n}\pi^{\bone}$ of the bundles $\pi^{\bzero} =
\pi^0$ and $\pi^{\bone}$ over the base $M^n$. We put $\Gamma(\pi) =
\Gamma^{\bzero}(\pi)\oplus\Gamma^{\bone}(\pi)$, where $\Gamma^{\bzero}(\pi) = \Gamma(\pi^{\bzero})$ and
$\Gamma^{\bone}(\pi) = \Pi (\Gamma(\Pi\pi^{\bone}))$;
this construction of~$\Gamma^{\bone}(\pi)$ in~$\Gamma(\pi)$ will be referred to in the definition of Cartan's distribution~$\cC$ (see~below).

Consider the jet space $J^{\infty}(\pi)$ of sections of
the super%
bundle~$\pi$. Namely, for the superbundle~$\pi$ the infinite jet superbundle $\pi_\infty\colon J^\infty(\pi)\to
M^n$ is defined as follows: we let $(\pi_\infty)^{\bzero} = (\pi^{\bzero})_\infty$, $(\pi_\infty)^{\bone} = \Pi(
(\pi^1)_\infty)$ (see~\cite{JKLStarCov} and~\cite{Norway} for details).
The set of variables describing %
$J^{\infty}(\pi)$ is composed by
\begin{itemize}
\item even coordinates $x^i$ on $M^{n}$,
\item even coordinates $u^j$ and parity\/-\/odd coordinates~$\xi^k$ 
along the fibres of~$\pi$; these objects %
themselves are elements of the set of
\item even variables %
$u^j_{\sigma}$ and parity\/-\/odd variables~$\xi^k_{\sigma}$ for %
the fibres of the infinite jet bundle $\pi_\infty\colon J^{\infty}(\pi)\to M^n$.
\end{itemize}
In the above notation we let $\sigma$~be the multi\/-\/index that
labels partial derivatives of the unknowns $u^j$ and~$\xi^k$
w.r.t.\ the even variables $x^i$; by convention, 
$u^j_{\varnothing}\equiv u^j$ and $\xi^k_{\varnothing}\equiv \xi^k$.
The parity function $\parity$ acts \emph{only} on homogeneous elements of $C^\infty(J^\infty(\pi))$ by extending its value from the generators,
\begin{align*}
\parity(x^i) & =  \bzero, &
\parity(u^j) & =  \bzero, & \parity(\xi^k) & =  \bone, && \\
&& \parity(u^j_{\sigma}) &= \bzero, & \parity(\xi^k_{\sigma}) & = \bone, &
|\sigma|&>0.
\end{align*}
By construction, the parity satisfies the rules
\begin{align*}
\parity(a \cdot b) &= \parity(a) + \parity(b),\\
\parity(a + b)&=\parity(a) = \parity(b) \ \text{ iff } \parity(a) = \parity(b),
\end{align*}
where $a,b\in C^\infty(J^\infty(\pi))$.

Every fibrewise linear function~$f\in C^{\infty}_{\text{lin}}(J^\infty)$ can be identified naturally with a
linear differential operator $\Delta_f\colon \Gamma(\pi) \to C^{\infty}(M)$ by using the formula
$\Delta_f(s)(x) = f(j_{\infty}(s)(x))$, where $j_{\infty}(s)(x)$ is the infinite jet of a
section~$s\in\Gamma(\pi)$ at~$x\in M$.  The infinite jet bundle~$\pi_{\infty}$ admits a natural flat
connection such that the lift~$\hat{X}$ of a vector field $X$ on $M$ is uniquely defined by the condition
$\Delta_{\hat{X}(f)} = X\circ\Delta_f$ for $f\in C^\infty_{\text{lin}}(J^\infty(\pi))$. The
lifts~$D_{x^i} = \widehat{{\dd}/{\dd x^i}}$ of~${\dd}/{\dd x^i}$ are called \emph{the total derivatives} on
$J^\infty(\pi)$; at every~$i$, they are expressed by the formula
\[%
D_{x^i} = \frac{\partial}{\partial x^i} +
  \sum^{m_0}_{j=1}\sum_{\sigma_{\bzero}}
  u^j_{\sigma_{\bzero}+1_i}\frac{\partial}{\partial u^j_{\sigma_{\bzero}}}
  + \sum^{m_1}_{k=1}\sum_{\sigma_\bone}
  \xi^k_{\sigma_\bone+1_i}\frac{\vec{\partial}}{\partial \xi^k_{\sigma_\bone}},
\]%
where ${\vec{\dd}}/{\dd\xi^k}$ denotes the left derivative.
These vector fields commute (in a usual sense, even though the operators~$D_{x^i}$ contain %
directed derivations).
By definition, we put $D_\tau=D_{x^1}^{\tau_1}\circ\dots\circ D_{x^n}^{\tau_n}$. Vector fields of the form
$\hat{X}$ generate an $n$-dimensional distribution on~$J^{\infty}(\pi)$; it is called the Cartan distribution
and it is denoted by~$\cC$.

\subsection{Differential equations}
In this $\BBZ_2$-\/graded set\/-\/up, let a system~$\cE$ of partial differential equations be given. By definition, the geometric object~$\cE$ is described by (many equivalent) systems of relations between the unknowns' derivatives with respect to the $n$ independent directions along the base~$M^n$. In local coordinates we have that
\[
\cE = \left\{ F^\ell(x^i, u^j, \dots, u^j_{\sigma_\bzero}, 
\xi^k, \dots, \xi^k_{\sigma_\bone}) = 0, 
\qquad \ell=1, \ldots, r \right\}.
\]
In fact, not every object~$\cE$ determined this way would be interesting from either geometric or physical points of view. To get rid of %
irrelevant cases, from now on we consider only ($\BBZ_2$-\/graded) partial differential (super-)\/equations which are \emph{formally integrable} in the sense of Goldschmidt~\cite{GoldschmidtLin,GoldschmidtNonlin}.
Still let us emphasise that at the moment when this paper is written, the expert community has not yet reached a consensus on the proper $\BBZ_2$-\/\emph{graded} extension of Goldschmidt's classical result on integrability. It is quite paradoxical that even if such extension is almost as straightforward as the generalisation of Marvan's approach to kinematically integrable systems, that work has not yet been done. 

In view of what has been said above, we accept that each partial differential equation to study must possess
the non\/-\/empty infinite prolongation~$\cEinf$ formed by all the differential consequences\footnote{An
  obvious logical and geometric distinction between the locus~$\cEinf$ and its algebraic description by using
  the $C^\infty$-\/smooth left\/-\/hand sides in the system $D_\tau(F^\ell)=0$ is that the latter are always
  defined yet they can describe the \emph{empty set}. For instance, consider the overdetermined partial
  differential equation $\cE=\{u_{xx}=1$, $u_y=x^2\}$ for which $(u_{xx})_y=0\neq 2=(u_y)_{xx}$. Likewise, the
  equation $\cE=\{v_x=u$,\ $v_y=u\}$ can be solved only if the compatibility condition~$v_{xy}=v_{yx}$ is
  satisfied, thus $u_x=u_y$ is the constraint due to which the projection of~$\cEinf$ down to~$\cE$ is not
  surjective.}  $D_\tau(F^\ell)=0$ with $|\tau|\geqslant0$; 
the locus\footnote{%
Neither the set~$\cE\subseteq J^k(\pi)$ nor its prolongation~$\cE^\infty\subseteq J^\infty(\pi)$ may be expected to be submanifolds in the respective jet spaces. For example, consider the differential equation $\cE=\{u_x^2=u^2\}\subset J^1(\pi)$ which cuts the diagonal cross (i.e.\ already not a submanifold) in the coordinate plane $Ouu_x$ within~$J^1(\pi)$. (It is clear also that the set of solutions to the differential equation $\bigl(\tfrac{\Id}{\Id x}{\bigr|}_{x_0-0} u\bigr)^2=\bigl(u(x_0)\bigr)^2$ on~$M^1=\BBR\ni x_0$ is immense, compared with the solution sets for the equations $u_x=u$ and~$u_x=-u$.) Moreover, should there be \emph{two} independent variables, $x$ and~$t$, so that $\cE=\{u_x^2=u^2\}$ is a \emph{partial} differential equation, then it is readily seen that, parameterised by using infinitely many variables $x$,\ $t$,\ $u$,\ $u_t$,\ $u_{tt}$,\ $\ldots$,\ $u_{t\cdots t}$,\ $\ldots$, the locus~$\cE^\infty$ is not a submanifold in~$J^\infty(\pi)$ as~well.}
$\cEinf\subseteq J^\infty(\pi)$ is
required to %
project back onto~$\cE$ and onto all the lower\/-\/order jet (super-)\/spaces~--- so that the Cauchy problem
for~$\cE$ is (formally) solvable in the class of formal power series for all Cauchy data.

Strange though it may seem, the social request for a %
$\BBZ_2$-\/graded generalisation of Marvan's removability inspection method offers us the most restrictive requirement for the class of PDEs to be studied: they must possess Lie (super-)\/algebra valued zero\/-\/curvature representations~--- and even parametric families~$\alpha_\lambda$ of such structures.

\subsection{Differential forms}
Let us denote by~$\bar{D}_{x^i}$ the restrictions of total derivatives~$D_{x^i}$ to~$\cEinf\subseteq J^\infty(\pi)$. 
At every point $\theta^\infty\in\cEinf$ the tangent 
space~$T_{\theta^\infty}\cEinf$ splits in a direct sum of two subspaces. The one which is spanned by the Cartan distribution~$\cEinf$ is \emph{horizontal} and the other is \emph{vertical}:
$T_{\theta^\infty} \cEinf = \cC_{\theta^\infty} \oplus
V_{\theta^\infty} \cEinf$. 
We denote by~$\Lambda^{1,0}(\cEinf) = \Ann \cC$ 
and~$\Lambda^{0,1} (\cEinf) = \Ann V\cEinf$ 
the $C^\infty(\cEinf)$-\/modules of contact and horizontal
one\/-\/forms which vanish on~$\cC$ and~$V\cEinf$, respectively.
Denote further by~$\Lambda^r(\cEinf)$ the 
$C^{\infty}(\cEinf)$-\/module of $r$-forms on~$\cEinf$.
There is a natural decomposition $\Lambda^r(\cEinf) = \bigoplus_{q+p = r}
\Lambda^{p,q}(\cEinf)$, where $\Lambda^{p,q} (\cEinf) = \bigwedge^p \Lambda^{1,0}(\cEinf) \wedge \bigwedge^q \Lambda^{0,1}(\cEinf)$. This implies that the de Rham differential~$\bar{\Id}$ on~$\cEinf$ is subjected to the decomposition\label{pCartanDiff} 
$\bar{\Id} = \bar{\Id}_h + \bar{\Id}_{\cC}$, where $\bar{\Id}_h \colon \Lambda^{p,q}(\cEinf) \to \Lambda^{p,q+1}(\cEinf)$ is the horizontal differential and $\bar{\Id}_{\cC} \colon \Lambda^{p,q}(\cEinf) \to \Lambda^{p+1,q}(\cEinf)$ is the vertical differential.

The differential $\bar{\Id}_h$ can be expressed in coordinates by inspection of its action on elements of $C^\infty(\cEinf)=\Lambda^{0,0}(\cEinf)$:
for any function~$\phi$ we have that
\[%
  \bar{\Id}_h \phi = {} %
\sum_{i=1}^{n} \Id x^i \wedge \bar{D}_{x^i}(\phi),%
\qquad
  \bar{\Id}_{\cC} \phi = {} %
\sum_{j=1}^{m_0}\sum_{\sigma_\bzero} 
\omega^j_{\sigma_\bzero} \wedge \frac{\partial 
    \phi}{\partial u^j_{\sigma_\bzero}} 
+ \sum_{k=1}^{m_1}\sum_{\sigma_\bone} \zeta^k_{\sigma_\bone} 
\wedge \frac{\vec{\partial} \phi}{\partial \xi^j_{\sigma_\bone}},
\]%
where we put
\begin{align*}
\omega^j_{\sigma_\bzero} &= \Id u^j_{\sigma_\bzero} -
\sum_{j=1}^{n} u^j_{\sigma_\bzero+1_i}\,\Id x^i, & 
\zeta^j_{\sigma_\bone} &= \Id\xi^j_{\sigma_\bone} -
\sum_{i=1}^{n} \xi^j_{\sigma_\bone+1_i}\, \Id x^i.
\end{align*}
The horizontal differential $\bar{\Id}_h$ acts on the spaces $\Lambda^{p,q}(\cE^\infty)$ of differential forms via the graded Leibniz rule; its application to Cartan's forms $\bar{\Id}_{\cC}(u^j_\sigma)$ is deduced from the identity~$\bar{\Id}^2=0$ for the de Rham differential $\bar{\Id}=\bar{\Id}_h+\bar{\Id}_{\cC}$ on~$\cE^\infty$. Specifically, from $\bar{\Id}_h^2=\bar{\Id}_h\circ\bar{\Id}_{\cC}+\bar{\Id}_{\cC}\circ\bar{\Id}_h=\bar{\Id}_{\cC}^2=0$ one infers that 
$\bar{\Id}_h\circ\bar{\Id}_{\cC}=-\bar{\Id}_{\cC}\circ\bar{\Id}_h$, thus reducing the action of~$\bar{\Id}_h$ to the case where it has already been defined. %
The formula
$\bar{\Id}_h = \sum\nolimits_i \Id x^i \wedge {\bar{D}}_{x^i}$
now means that the vector fields $\bar{D}_{x^i}$ proceed by the Leibniz rule over the argument's wedge factors,
acting on each factor --\,pushed leftmost\,-- via the Lie derivative.

We note further that $\Id x^i$, $\Id u^j_{\sigma_\bzero}$, %
and $\Id \xi^k_{\sigma_\bone}$ %
satisfy the following commutation relations:
\begin{align*}
\Id x^i \wedge \Id x^j &{} = - \Id x^j \wedge \Id x^i, &
\Id x^i \wedge \Id u^j_{\sigma_\bzero} &{}= - \Id u^j_{\sigma_\bzero} \wedge \Id x^i, &
\Id x^i \wedge \Id \xi^k_{\sigma_\bone} &{}= -\Id \xi^k_{\sigma_\bone} \wedge \Id x^i, 
\\
\Id u^j_{\sigma_\bzero} \wedge \Id u^k_{\tau_\bzero} &{} = - \Id u^k_{\tau_\bzero} \wedge \Id u^j_{\sigma_\bzero}, &
\Id \xi^k_{\sigma_\bone} \wedge \Id u^j_{\tau_\bzero} &{}= -\Id u^j_{\tau_\bzero} \wedge \Id\xi^k_{\sigma_\bone}, &
\Id \xi^k_{\sigma_\bone} \wedge \Id \xi^j_{\tau_\bone} &{}= +\Id \xi^j_{\tau_\bone} \wedge \Id\xi^k_{\sigma_\bone};
\end{align*}
we refer to~\cite{gvbv,arXiv12100726%
} for the geometric theory of variations in the frames of which
one discovers why differential one\/-\/forms should anticommute in the $\mathbb{Z}$-\/graded sense.

The \emph{substitution} of a $\BBZ_2$-\/graded vector field~$X$ into 
a $\BBZ_2$-\/graded differential form~$\omega$ 
is defined by the formula $\mathrm{i}_X(\omega)=(-1)^{\parity(X)\cdot\parity(\omega)}\omega(X)$, provided that~$X$ and~$\omega$ are both homogeneous with respect to the $\BBZ_2$-\/grading. We have that
\[
\mathrm{i}_{\bar{D}_{x^i}}(\omega^j_{\sigma_\bzero}) = \mathrm{i}_{\bar{D}_{x^i}}(\zeta^k_{\sigma_\bone}) =0\qquad \text{for all $i,j,k$ and~$|\sigma|\geqslant0$}.
\]
These equalities mean that the Cartan distribution can be described equivalently 
in terms of the Cartan forms~$\omega^j_{\sigma_\bzero}$ and~$\zeta^k_{\sigma_\bone}$.

\subsection{Coverings over differentail equations} 
The restriction of Cartan's distribution from~$J^\infty(\pi)$ onto~$\cEinf$ 
is horizontal with respect to the projection
$\pi_{\infty}{\bigr|}_{\cEinf}\colon\cEinf\to M^n$. This determines the connection
$\cC_{\cEinf}\colon \Gamma (TM^n)\to\Gamma (T\cEinf)$, where $\Gamma (TM^n)$ and
$\Gamma (T\cEinf)$ are the $C^{\infty}(M^n)$-{} and $C^{\infty}(\cEinf)$-\/modules 
of vector fields on $M^n$ and $\cEinf$, %
respectively. 
We denote by $\Gamma T(\Lambda^1 (\cEinf))$ the
$C^{\infty}(\cEinf)$-\/module of derivations $C^{\infty}(\cEinf) \to \Lambda^1(\cEinf)$
taking values in the $C^{\infty}(\cEinf)$-\/module of one\/-\/forms~on~$\cEinf$.
The connection form $U_{\cEinf} \in \Gamma T(\Lambda^1
(\cEinf))$ of %
$\cC_{\cEinf}$ is called the \emph{structural element} of the equation~$\cEinf$, see~\eqref{EqConForm} on p.~\pageref{EqConForm}.

\begin{define}[\cite{BVV,JKAMNonlocalTrends}]\label{defCovering}
A \emph{covering} (or \emph{differential covering}) over a given partial differential (super-)\/equation $\cE$ is another (usually,
larger) system of partial differential equations $\tilde{\cE}$
endowed with the $n$-\/dimensional Cartan distribution $\tilde{\cC}$
and such that there is a mapping $\tau\colon\tilde{\cE}\to\cEinf$ for which at each point
$\theta\in\tilde{\cE}$, the tangent map $\tau_{*,\theta}$ is an isomorphism of the plane
$\tilde{\cC}_{\theta}$ to the Cartan plane $\cC_{\tau(\theta)}$ at the point~$\tau(\theta)$ in $\cEinf$.
\end{define}

The construction of a covering over $\cE$ means the
introduction of new variables such that their compatibility
conditions lie inside the initial system~$\cEinf$.
In practice (see~\cite{GDE2012} and references therein), it is the rules to
differentiate the new variable(s) which are specified in a consistent
way; this implies that those new variables acquire the nature of
nonlocalities if their derivatives are local but the variables
themselves are not (e.g., consider the potential $\mathfrak{v}=\int
u\,\Id x$ satisfying $\mathfrak{v}_x = u$ and $\mathfrak{v}_t = -u_{xx} -
3u^2$ for the KdV equation $u_t + u_{xxx} + 6uu_x = 0$).
Whenever the covering $\tau\colon\tilde{\cE}\to\cE$ is %
realised as a %
fibre bundle, the forgetful map~$\tau$ discards the nonlocalities.

\section{(Non)\/removability of %
parameters in $\BBZ_2$-\/graded zero\/-\/curvature representations}\label{GradedMarvan}
\noindent%
In this section we describe an %
algorithm for inspection of (non)removability of %
parameters in zero\/-\/curvature representations under the action of smooth families of gauge transformations. This technique (and its domain of applicability), which we formulate here for the $\BBZ_2$-\/graded set\/-\/up of super\/-\/equations and Lie super\/-\/algebras, patterns upon M.~Marvan's approach for the purely bosonic case~\cite{Marvan2002,Marvan2010}.
We recall that the latter works under the assumption of local analyticity for all the objects and structures involved. We now formulate the most essential half of Marvan's \emph{criterion} of (non)\/removability; to this end, we explicitly postulate that the admissible families of gauge transformations depend on the parameter in a smooth way. This covers the situations one typically encounters %
in mathematical physics; the case of smooth families of zero\/-\/curvature representations such that the parameter contained in them is removed by using the families of gauge transformations that are \emph{not smooth} is henceforth put aside.

\begin{rem}
Some agreement on the smoothness class of gauge transformations is always built into the concept of principal fibre bundles and gauge connection one\/-\/forms (e.g., those forms which satisfy the Maurer\/--\/Cartan flatness equation).
Even though the transformations of the wave function~$\Psi$ by elements~$S$ of the structure Lie {(super-)} %
group~$G$ are defined pointwise over the base manifold~$M^n\ni\boldsymbol{x}$ and therefore, they can be performed discontinuously with respect to the points~$\boldsymbol{x}$, this extent of generality is usually not the case~--- indeed, the gauge set\/-\/up is studied only under much more restrictive postulates. In particular, the introduction of gauge connection one\/-\/forms requires that both the wave function~$\Psi$ and gauge transformations~$S$ be {(piecewise-)} %
continuously differentiable.

Whenever the principal fibre bundles are towered over partial differential equations --\,in the context of Lie (super-)\/algebra $\fg$-\/valued zero\/-\/curvature representations~$\alpha$ and the inverse scattering\,-- the smoothness classes of such structures and their gauge transformations are determined from the smoothness classes of equations' solutions~$u^i=s^i(\boldsymbol{x})$ in the course of restriction of all the objects at hand to the jets~$j_\infty(s^i)(\boldsymbol{x})$ of solutions; the object~$\bar{\Id}_h S$ in~\eqref{gaugetrans} below would be undefined otherwise.

In conclusion, the choice of families~$S(\boldsymbol{x})\in G$ of gauge transformations for the fields~$\alpha %
\in\fg\mathbin{{\otimes}_{\Bbbk}}\Lambda^1(M^n)$ at~$\boldsymbol{x}\in M^n$ always refers, either explicitly or tacitly, to some \textit{ad hoc} assumptions on these families' and fields' smoothness. Our technical agreement that the families~$S_\lambda$ and~$\alpha_\lambda$ of \emph{such} structures both depend smoothly on a given parameter~$\lambda\in\mathcal{I}\subseteq\BBC$ fits into the general picture.
\end{rem}

Let $\cE$ be a partial differential (super-)\/equation whose infinite prolongation~$\cE^\infty$ is contained in the infinite jet (super-)\/bundle~$J^\infty(\pi)$; we refer to section~\ref{Notation} for a recollection of concepts and structures that arise in the geometry of jet (super-)\/spaces (e.g., we refer to that %
preliminaries chapter for the notions of the space $\bar{\Lambda}(\cEinf)$ of horizontal differential forms on~$\cEinf$ and the horizontal differential~$\bar{\Id}_h\colon\bar{\Lambda}(\cEinf)\to\bar{\Lambda}(\cEinf)$).

Let $G$ be a finite-dimensional matrix Lie supergroup (i.e., $G$ is a Lie supersubgroup of $GL(k_0| k_1)$
for certain non-negative integers $k_0$ and $k_1$). Let $\fg$ be its (matrix) Lie superalgebra.  Consider its %
tensor product $\fg\mathbin{{\otimes}_{\mathbb{R}}}\bar{\Lambda}(\cEinf)$ 
with the exterior algebra
$\bar{\Lambda}(\cEinf)=\bigoplus_i \Lambda^{0,i}(\cEinf)$;
by definition, elements of $\fg\otimes C^\infty(\cEinf)$ are called
\emph{$\fg$-\/\textup{(\emph{super})}\/matrices}~\cite{Marvan2002}.

The product is endowed with the bracket (see p.~\pageref{pParity} for definition of the parity function~$\parity$)
 \[\bbl A\otimes\mu, B\otimes\nu \bbr =
 (-1)^{\parity(B)\parity(\mu)}\bbl A,B\bbr\otimes\mu\wedge\nu\] 
for
$\mu,\nu\in\bar{\Lambda}(\cEinf)$ and $A,B\in \fg$.
Define the operator $\bar{\Id}_h$ that acts on elements of $\fg\otimes\bar{\Lambda}(\cEinf)$ by the rule 
\[
\bar{\Id}_h(A\otimes\mu) = A\otimes\bar{\Id}_h\mu,
\]
where $\bar{\Id}_h$ in the right\/-\/hand side is the horizontal differential.
The tensor product ${\fg\otimes\bar{\Lambda}(\cEinf)}$ is a
differential graded associative
algebra with respect to the multiplication $(A\otimes\mu)\cdot (B\otimes\nu) =(-1)^{\parity(B)\parity(\mu)} (A\cdot
B)\otimes\mu\wedge\nu$ induced by the ordinary matrix multiplication
so that
\begin{align*}
  \bbl \rho,\sigma \bbr &= \rho\cdot\sigma -
    (-1)^{rs}(-1)^{\parity(\rho)\parity(\sigma)}\sigma\cdot\rho,\\
  \bar{\Id}_h(\rho\cdot\sigma) &= \bar{\Id}_h\rho\cdot\sigma + (-1)^r\rho\cdot\bar{\Id}_h\sigma
\end{align*}
for $\rho\in\fg\otimes\bar{\Lambda}^r(\cEinf)$ and
$\sigma\in\fg\otimes\bar{\Lambda}^s(\cEinf)$.

\begin{define}[\cite{Marvan2002,Marvan2010,Manin84}]
A horizontal $1$-\/form
$\alpha\in\fg\otimes\bar{\Lambda}^1(\cEinf)$ is
called a $\fg$-\/valued \emph{zero\/-\/curvature representation} for
the equation~$\cE$ if the Maurer\/--\/Cartan condition,
\begin{equation}\label{zcrf}
\bar{\Id}_h\alpha \doteq \tfrac12\bbl \alpha,\alpha \bbr,
\end{equation}
holds whenever both sides are restricted to~$\cE$ and its differential consequences
(such restriction is denoted by~$\doteq$).
\end{define}

Recall that $\fg$ is the Lie superalgebra of a given Lie supergroup~$G$.
Elements of the (pre-)\/sheaf~$C^{\infty}(\cEinf, G)$ of $G$-valued functions on the equation~$\cEinf$ are
called \emph{$G$-\/matrices}. One can represent elements of $C^{\infty}(\cEinf, G)$ as block matrices
$\left(\begin{smallmatrix} A & B \\ C & D \end{smallmatrix}\right)$,
where $A$~is a $(k_0\times k_0)$-\/size matrix whose entries are even
elements of~$C^{\infty}(\cEinf)$, 
$B$~is a $(k_0\times k_1)$-\/size matrix whose entries are odd elements
of~$C^{\infty}(\cEinf)$, 
$C$~is a $(k_1\times k_0)$-\/size matrix whose entries are odd elements
of~$C^{\infty}(\cEinf)$, 
and $D$~is a $(k_1\times k_1)$-\/size matrix whose entries are even elements
of~$C^{\infty}(\cEinf)$.

\begin{define}\label{defNonRem}
Let $\alpha$ and $\alpha^\prime$ be $\fg$-valued zero\/-\/curvature representations. Then $\alpha$ and $\alpha^\prime$ are
called \emph{gauge\/-\/equivalent} if there exists $S\in C^{\infty}(\cEinf, G)$ such that 
\begin{equation}\label{gaugetrans}
  \alpha^{\prime}  = \bar{\Id}_hS\cdot S^{-1} + S\cdot\alpha\cdot S^{-1}
\mathrel{{=}{:}} \alpha^S. %
\end{equation}
Let $\alpha_{\lambda}$ be a family of zero\/-\/curvature representations smoothly depending on a %
parameter~$\lambda\in\mathcal{I}\subseteq\BBR$ (or $\subseteq\BBC$), 
where $\mathcal{I}$ is a connected subset. %
The parameter $\lambda$ is \emph{removable} 
under the action of a smooth family of gauge transformations
if there exists $\lambda_0\in\mathcal{I}$ and there is a family of
gauge transformation $S_\lambda$ smoothly depending on $\lambda$ such that 
$\alpha_{\lambda}  = \alpha_{\lambda_0}^{S_\lambda}$ for all $\lambda\in\mathcal{I}$.
\end{define}

\begin{rem}
There are other approaches to the idea of parameters' (non)\/removability, e.g.,
under transformations which are not necessarily gauge 
(this is in contrast to the above definition). 
It turns out that a given parameter in a family of zero\/-\/curvature
representations can be nonremovable with respect to the %
class of smooth gauge transformations
but, at the same time, it can be eliminated by using transformations from a wider group.
For example, 
Sasaki showed in~\cite{Sasaki79} that the parameter in the standard Lax pair for
the Korteweg\/--\/de Vries equation 
cannot be gauged out but it can be {\it eliminated} by using the scaling symmetry of~KdV
(see Appendix~\ref{appParamNotes} in this paper).
We stress that Sasaki's %
transformation is not gauge and therefore it acts across the 
gauge group orbits;
that parameter is \emph{non\/-\/removable} in the sense of Definition~\ref{defNonRem}
because there is no smooth family of gauge transformation which would remove it.  
We refer to~\cite{Marvan2010} for a discussion
about parameters that can be removed by gauge transformation %
depending on the parameter~$\lambda$ not (only) in a smooth way.
We note that in the most interesting examples the nonremovable parameter could be eliminated,
see Appendix~\ref{appParamNotes} and~\cite{JKIgonin2002} for examples.
\end{rem}

\begin{prop}[cf.\ \cite{Marvan2002}]\label{PropNonRemGrad}
Let $\fg$~be a complex matrix Lie \textup{(}super-\textup{)}\/algebra.
  For a connected subset $\mathcal{I}\subseteq\BBC$, consider a family~$\alpha_{\lambda}$, depending smoothly
  on a parameter~$\lambda\in\mathcal{I}$, of $\fg$-valued zero\/-\/curvature representations for an
  equation~$\cE$. If for each $\lambda\in\mathcal{I}$ there is a $\fg$-\/matrix~$Q_{\lambda}$ such that its
  parity is $\parity(Q_\lambda)=\bzero$ and
\begin{equation}\label{EqToSolve}
  \frac{\partial}{\partial \lambda}\alpha_{\lambda} =
  \bar{\Id}_h Q_{\lambda} - \bbl \alpha_{\lambda}, Q_{\lambda}\bbr,
\end{equation}
then the parameter~$\lambda$ is removable under the action of a %
smooth family of gauge transformations.
\end{prop}

\begin{proof}
Suppose %
that 
$\dot{\alpha}_{\lambda}=\bar{\Id}_hQ_{\lambda} - \bbl {\alpha_\lambda}, Q_{\lambda} \bbr$ for some
$Q_{\lambda}\in\fg\otimes C^{\infty}(\cEinf)$. To begin with, fix a constant $\lambda_0\in\mathcal{I}$. Let
$S_{\lambda}\in C^\infty(\cE^\infty,G)$ be a solution of the matrix equation%
\footnote{We recall that $\fg$~is the matrix Lie (super-)\/algebra
of a given matrix Lie (super-)\/group~$G$, whence the multiplication $Q_{\lambda}\cdot S_{\lambda}$ is induced by the
ordinary multiplication of (super-)matrices.}
$\partial S_\lambda/\partial\lambda = Q_{\lambda}\cdot S_{\lambda}$ with the initial datum
$S_{\lambda_0} = E$, where $E$ is the identity matrix $\left(\begin{smallmatrix} \boldsymbol{1} & 0 \\ 0 &
  \boldsymbol{1} \end{smallmatrix}\right) \in \Mat(k_0| k_1)$. 
Consider the expression $Z_{\lambda}= \bar{\Id}_h S_\lambda + S_{\lambda} \alpha_{\lambda_0} -
\alpha_{\lambda} S_{\lambda} = (\alpha_{\lambda_0}^{S_{\lambda}} - \alpha_\lambda) S_\lambda$.  We have that
\begin{align*}
\frac{\partial}{\partial\lambda}Z_\lambda & =
   \frac{\partial}{\partial\lambda} ( \bar{\Id}_h S_\lambda +
S_{\lambda} \alpha_{\lambda_0} - \alpha_{\lambda} S_{\lambda} )\\
{} & =  \bar{\Id}_h(\dot{S}_\lambda) + \dot{S}_{\lambda}\alpha_{\lambda_0} -
  \dot{\alpha}_\lambda S_\lambda - \alpha_\lambda \dot{S}_\lambda\\
{} & = \bar{\Id}_h(Q_\lambda S_\lambda) + Q_\lambda S_\lambda \alpha_{\lambda_0}
  - \dot{\alpha}_\lambda S_\lambda - \alpha_\lambda Q_\lambda
  S_\lambda = \\
{} & = \bar{\Id}_h Q_\lambda S_\lambda + Q_\lambda \bar{\Id}_h
  S_\lambda  + Q_\lambda S_\lambda \alpha_{\lambda_0}
  - \dot{\alpha}_\lambda S_\lambda 
  - \alpha_\lambda Q_\lambda S_\lambda
  + \left( Q_\lambda\alpha_\lambda S_\lambda
  - Q_\lambda\alpha_\lambda S_\lambda \right)
 \\
{} & = (\underbrace{\bar{\Id}_h Q_\lambda - \alpha_\lambda Q_\lambda +
  Q_\lambda 
  \alpha_\lambda}_{%
  \dot{\alpha}_\lambda}  - \dot{\alpha}_\lambda)S_\lambda + Q_\lambda
( \bar{\Id}_h 
  S_\lambda + S_\lambda \alpha_{\lambda_0} - \alpha_\lambda S_\lambda)
  \\
{} & = Q_\lambda Z_\lambda.
\end{align*}
It is obvious that $Z_{\lambda_0} = 0$, whence
$\alpha_{\lambda_0}^{S_\lambda} - \alpha_\lambda = 0$.
Therefore, %
the parameter~$\lambda$ is removable
by using the explicitly given family~$S_\lambda$ of gauge transformations.
\end{proof}

The %
above proposition and its proof for \emph{parity\/-\/even} $\fg$-\/matrices~$Q_\lambda$ are a %
straightforward %
$\BBZ_2$-\/graded generalisation of
Marvan's original idea %
for %
non\/-\/graded PDE systems%
~\cite{Marvan2002,Marvan2010};
the commutator $[\cdot,\cdot]$ in a Lie algebra is now replaced %
by the graded commutator $\bbl \cdot, \cdot \bbr$ in the Lie superalgebra.
Note that whenever $\parity(Q_\lambda)=\bzero$,
the parity $\parity(Q_\lambda)+\parity(S_\lambda)$ of $\dot{S}_\lambda$
in the right\/-\/hand side of the equation $\dot{S}_\lambda=Q_\lambda S_\lambda$
is the same as that of $S_{\lambda}$,
which agrees with~$\parity(\lambda)=\bzero$.

\begin{rem}\label{RemAlmostConverse}
  It is readily seen now \emph{what} could obstruct the converse to be true, which would otherwise convert
  Proposition~\ref{PropNonRemGrad} into the criterion of (non)\/removability. Unfortunately, the
  family~$S_\lambda$ of gauge transformations may not necessarily be \mbox{(piecewise-)}\/smooth in~$\lambda$ even if
  the family~$\alpha_\lambda$~is.

Suppose still that the parameter~$\lambda$ in~$\alpha_\lambda$ is removable by a smooth family of $G$-matrices~$S_\lambda$ (so that the derivatives~$\dot{S}_\lambda$ with respect to~$\lambda$ are well defined).
This means that for any fixed $\lambda_0$ there exists a $G$-\/matrix $S_{\lambda}$
such that $\alpha^{S_{\lambda}}_{\lambda_0} = \alpha_{\lambda}$ with
$S_{\lambda_0} = E\in G$, which is viewed here as the set of constant $G$-valued functions on~$\cEinf$.
The matrix $\dot{S}_{\lambda_0} = \partial/\partial\lambda |_{\lambda=\lambda_0} S_{\lambda}$ belongs to the
tangent space at the unit of~$G$, i.e.\ to the matrix Lie
(super-)\/algebra~$\fg$. %
We have that
\begin{align*}
0 &=
    \left.\frac{\partial}{\partial\lambda}\right|_{\lambda=\lambda_0}
      \alpha_{\lambda_0} =
    \left.\frac{\partial}{\partial\lambda}\right|_{\lambda=\lambda_0}
      \alpha^{S^{-1}_{\lambda}}_{\lambda} =
    \left.\frac{\partial}{\partial\lambda}\right|_{\lambda=\lambda_0}
      \left(\bar{\Id}_h (S^{-1}_{\lambda})S_\lambda + S_{\lambda}^{-1}\alpha_\lambda
      S_\lambda\right)= {}
\\
{}&=\left.\frac{\partial}{\partial\lambda}\right|_{\lambda=\lambda_0}
      \left(
        -S^{-1}_{\lambda}\bar{\Id}_h S_{\lambda} +
        S_{\lambda}^{-1}\alpha_{\lambda} S_{\lambda}
      \right) 
= {}\\
{}&=
-\frac{\partial}{\partial\lambda} \bigl(S^{-1}_{\lambda_0}\bigr)\, \bar{\Id}_h (S_{\lambda_0})  - S_{\lambda_0}^{-1}
   d\dot{S}_{\lambda_0} -
   S_{\lambda_0}^{-1}\dot{S}_{\lambda_0}S_{\lambda_0}^{-1}\alpha_{\lambda_0}S_{\lambda_0}
\\
{}&{}\qquad\qquad  + S_{\lambda_0}^{-1}\dot{\alpha}_{\lambda_0}S_{\lambda_0} +
   S_{\lambda_0}^{-1}\alpha_{\lambda_0}\dot{S}_{\lambda_0} = 
   - \bar{\Id_h}\dot{S}_{\lambda_0} - \dot{S}_{\lambda_0}\alpha_{\lambda_0} +
   \alpha_{\lambda_0}\dot{S}_{\lambda_0} + \dot{\alpha}_{\lambda_0}.
\end{align*}
This implies that %
$\dot{\alpha}_{\lambda_0} =   \bar{\Id}_h \dot{S}_{\lambda_0} - 
\bbl \alpha_{\lambda_0},  \dot{S}_{\lambda_0} \bbr$ 
for all~$\lambda_0$ in~$\mathcal{I}$, 
where $\dot{S}\in\fg\otimes\bar{\Lambda}^0(\cEinf)$.
\end{rem}

Let us illustrate the technique now offered by Proposition~\ref{PropNonRemGrad} by proving that the parameter in the Das zero\/-\/curvature representation for the Mathieu $N=2$, $a=4$ super\/-\/KdV equation is \emph{essential}.

Consider P.~Mathieu's $(2|2)$-\/dimensional %
generalisation of the Korteweg\/--\/de Vries equation~\cite{MathieuNew}, namely, the $N{=}2$ supersymmetric KdV %
equation (SKdV),
\begin{equation}\label{SKdV}
\bu_t=-\bu_{xxx}+3\bigl(\bu\cD_1\cD_2\bu\bigr)_x
 +\frac{a-1}{2}\bigl(\cD_1\cD_2\bu^2\bigr)_x + 3a\bu^2\bu_x,\qquad
 \cD_i=\frac{\vec{\dd}}{\dd\theta_i}+\theta_i\cdot\,\bar{D}_x, %
\end{equation}
where
\begin{equation}\label{N2superfield}
\bu(x,t;\theta_1,\theta_2)
=u_0(x,t)+\theta_1\cdot u_1(x,t)+\theta_2\cdot
u_2(x,t)+\theta_1\theta_2 \cdot u_{12}(x,t)
\end{equation}
is the complex bosonic super\/-\/field,
$\theta_1,\theta_2$ are Grassmann variables such that $\theta_1^2=\theta_2^2=\theta_1\theta_2 + \theta_2\theta_1 = 0$, 
the two fields $u_0$ and $u_{12}$ are bosonic %
($\parity(u_0) = \parity(u_{12}) = \bzero$), 
and the fields $u_1$ and $u_2$ are
fermionic  ($\parity(u_1) = \parity(u_2) = \bone$). Expansion~\eqref{N2superfield}
converts~\eqref{SKdV} to the four\/-\/component system 
\begin{subequations}\label{SKdVComponents}
\begin{align}
u_{0;t}&=-u_{0;xxx}+\bigl(a u_0^3
   -(a+2)u_0u_{12}+(a-1)u_1u_2\bigr)_x,
\label{GetmKdV}\\
u_{1;t}&=-u_{1;xxx}+\bigl(\phantom{+}(a+2)u_0u_{2;x}+(a-1)u_{0;x}u_2
   -3u_1u_{12}+3a u_0^2u_1 \bigr)_x,\\
u_{2;t}&=-u_{2;xxx}+\bigl(-(a+2)u_0u_{1;x}-(a-1)u_{0;x}u_1
   -3u_2u_{12}+3a u_0^2u_2 \bigr)_x,\\
\underline{u_{12;t}}&=\underline{-u_{12;xxx}-6u_{12}u_{12;x}}
 +3au_{0;x}u_{0;xx}+(a+2)u_0u_{0;xxx}\notag\\
 {}&{}\qquad{}+3u_1u_{1;xx}+3u_2u_{2;xx}
 +3a\bigl(u_0^2u_{12} -2u_0u_1u_2\bigr)_x.\label{GetKdV}
\end{align}
\end{subequations}
The Korteweg\/--\/de Vries equation upon~$u_{12}$, see~\eqref{kdv} below, is underlined in~\eqref{GetKdV}.
The SKdV equation is the most interesting (in particular, bi\/-\/Hamiltonian, whence completely integrable) if $a\in\{-2, 1,4\}$,
see~\cite{MathieuNew,HKKW,KisHus09}.
Let us consider also the bosonic limit $u_1=u_2=0$ of
system~\eqref{SKdVComponents}: %
by setting
$a{=}-2$ we obtain the triangular system which consists of the modified
KdV equation for $u_0$ and an equation of KdV\/-\/type for~$u_{12}$; in the
case $a{=}1$ we obtain the Krasil'shchik\/--\/Kersten system~\cite{Cyprus2014};
for $a{=}4$, we obtain the third equation in the Kaup\/--\/Boussinesq hierarchy. 
In what follows we consider the case~$a{=}4$.

\begin{example}\label{exN2ZCRDas}
The $N{=}2$ supersymmetric $a{=}4$-\/KdV equation~\eqref{SKdVComponents} admits
the $\gsl(2|1)$-\/valued zero\/-\/curvature representation 
$\alpha^{N=2}(\veps)= A(\veps)\,\mathrm{d}x + B(\veps)\,\mathrm{d}t$,
where
\begin{equation}\label{eqDasA}
  A  =
  \begin{pmatrix}
    -\ii u_0 &
    \veps^{-1} (u_{0}^2 + u_{12} ) - \ii\veps^{-2}u_{0}  &
      - \veps^{-1} (u_2 + \ii u_1)\\
      -\veps &
      -\ii u_0 - \veps^{-1} &
      0 \\
      0 &
      \ii u_1 - u_2 &
      -2\ii u_0 - \veps^{-1}
    \end{pmatrix}, \qquad 
    \veps > 0.
  \end{equation}
The elements of the $\gsl(2| 1)$-\/matrix $B$, %
\[
B = \begin{pmatrix}
  b_{11} & b_{12} & b_{13} \\
  b_{21} & b_{22} & b_{23} \\
  b_{31} & b_{32} & b_{33}
\end{pmatrix},
\]
are as follows (see Remark~\ref{RemDasUpToS} below),
\begin{align*}
b_{11} ={} & 4 \ii u_{0}^3 - 6 \ii u_{0} u_{12}  + 4u_{0} u_{0;x} -  \ii u_{0;xx}
- u_{12;x} - 4 \ii u_{2}u_{1} + \veps^{-1} (2u_{0}^2 - u_{12}  -  \ii u_{0;x})
-  \ii \veps^{-2}u_{0} ,\\
b_{12}={} & \veps^{-1}(4u_{0}^4 + 2u_{0}^2u_{12}  + 4u_{0} u_{0;xx} -
2u_{12}^2 + 4u_{0;x}^2 - u_{12;xx} + u_{2}u_{2;x} + 8u_{2}u_{1}u_{0}
+ u_{1}u_{1;x}) + {}\\
{}&{}+ \veps^{-2}(2 \ii u_{0}^3 - 4 \ii u_{0} u_{12}  + 4u_{0}
u_{0;x} -  \ii u_{0;xx} - u_{12;x} - 2 \ii u_{2}u_{1}) + \veps^{-3} (u_{0}^2 -
u_{12}  -  \ii u_{0;x})-{}\\
{}&{} -  \ii \veps^{-4}u_{0} ,\\
b_{13} = {} & \veps^{-1}( - 5 \ii u_{0} u_{2;x} - 5u_{0} u_{1;x} -
u_{2;xx} +  \ii u_{1;xx} + 8u_{2}u_{0}^2 - 2u_{2}u_{12}  - 4 \ii u_{2}u_{0;x}
- 8 \ii u_{1}u_{0}^2 + {}\\
{}&{}  + 2 \ii u_{1}u_{12}  - 4u_{1}u_{0;x}) + \veps^{-2} ( - u_{2;x} +
 \ii u_{1;x} - 3 \ii u_{2}u_{0}  - 3u_{1}u_{0} ) + \veps^{-3}( - u_{2} +
 \ii u_{1}),\\
b_{21} = {} & 2\veps( - 2u_{0}^2 + u_{12} ) + 2 \ii  u_{0}  +
\veps^{-1},\\
b_{22} = {} & 4 \ii u_{0}^3 - 6 \ii u_{0} u_{12}  - 4u_{0} u_{0;x} -
 \ii u_{0;xx} + u_{12;x} - 4 \ii u_{2}u_{1} + \veps^{-1}( - 2u_{0}^2 + u_{12}
+  \ii u_{0;x}) + {} \\
{}&{}+ \ii \veps^{-1} u_{0}  + \veps^{-3},\\
b_{23} = {} &  u_{2;x} -  \ii u_{1;x} + 4 \ii u_{2}u_{0}  + 4u_{1}u_{0}  +
\veps^{-1}(u_{2} -  \ii u_{1}),\\
b_{31} = {} & \veps ( - u_{2;x} -  \ii u_{1;x} + 4 \ii u_{2}u_{0}  -
4u_{1}u_{0} ) + u_{2} +  \ii u_{1}, \\
b_{32} = {} & 5 \ii u_{0} u_{2;x} - 5u_{0} u_{1;x} - u_{2;xx} -  \ii u_{1;xx}
+ 8u_{2}u_{0}^2 - 2u_{2}u_{12}  + 4 \ii u_{2}u_{0;x} + 8 \ii u_{1}u_{0}^2 -
2 \ii u_{1}u_{12}  - {} \\
{}&{} - 4u_{1}u_{0;x} + \veps^{-1} u_{0} ( \ii u_{2} - u_{1}),\\
b_{33}={} & 2(4 \ii u_{0}^3 - 6 \ii u_{0} u_{12}  -  \ii u_{0;xx} - 4 \ii u_{2}u_{1}) + \veps^{-3}.
\end{align*}
\end{example}
\begin{prop}\label{PropNonRemDas}
There is no $\gsl(2|1)$\/-matrix
\[
Q = \begin{pmatrix}
  q_{11} & q_{12} & q_{13} \\
  q_{21} & q_{22} & q_{23} \\
  q_{31} & q_{32} & q_{11} + q_{22}
 \end{pmatrix} 
 \]
which would depend on~$\veps$ and satisfy the equalities
\begin{subequations}\label{eqDasQ}
  \begin{align}
    \bar{D}_x(Q) = {}&{} \tfrac{\partial}{\partial \veps} A + \bbl A, Q \bbr,\label{eqDasQA}\\
    \bar{D}_t(Q) = {}&{} \tfrac{\partial}{\partial \veps} B + \bbl B, Q \bbr.
  \end{align}
\end{subequations}
In other words, the %
parameter $\veps$ in~$\alpha^{N=2}(\veps)$ cannot be removed by using a smooth family of gauge transformations.
\end{prop}

An %
analytic proof of Proposition~\ref{PropNonRemDas} is contained in Appendix~\ref{AppClaimZCRDas}, see p.~\pageref{AppClaimZCRDas}.

\begin{rem}\label{remDasJMP12}\label{RemDasUpToS}
This zero\/-\/curvature representation~$\alpha^{N=2}(\veps)$ is not equal identically
but it is gauge\/-\/equivalent to the respective formula in Das \textit{et al.}~\cite{Das}.
The transformation between these objects contains the imaginary unit~$\ii$.
Our choice of normalisation, which is the same as in~\cite{JMP2012}, is 
due to the following argument: all the structures under study contain the Gardner
deformation~\eqref{eqCovGardnerKdV} of the Korteweg\/--\/de Vries equation~\eqref{kdv}
(so that the structures retract %
to Gardner's deformation~\cite{Miura68,Gardner} under suitable reductions).

We note further that the zero\/-\/curvature representation $\alpha^{N{=}2}$ can be used 
for construction of a solution, which is an alternative to the first solution 
reported in~\cite{HKKW}, of Gardner's deformation
problem~\cite{MathieuNew,MathieuOpen} for the $N{=}2$, $a{=}4$ SKdV equation
(we refer to~\cite{JMP2012} for details). 
The parameter $\veps$ which we use here is the parameter in the classical 
Gardner deformation of the KdV equation~\cite{Miura68}. %
This is why we denote this parameter by $\veps$ instead of~$\lambda$.
\end{rem}

\begin{example}\label{exN2ZCRRem}
Consider now another $\gsl(2|1)$-\/valued zero\/-\/curvature representation $\beta =
A\,\mathrm{d}x + B\,\mathrm{d}t$ for the $N{=}2$,\ $a{=}4$-\/SKdV equation: we let
\[
A = \begin{pmatrix}
  \lambda - \ii u_0 & 
    - \lambda^2 - (u_0^2 + u_{12}) &
      - \ii u_1 - u_2 \\
  1 &
    - \lambda - \ii u_0 &
      0 \\
  0 &
    u_2 - \ii u_1 &
      - 2\ii u_0
\end{pmatrix}.
\]
The elements of the $\gsl(2| 1)$-\/matrix $B$, %
\[
B = \begin{pmatrix}
b_{11} & b_{12} & b_{13}  \\
b_{21} & b_{22} & b_{23}  \\
b_{31} & b_{32} & b_{33}
\end{pmatrix},
\]
are given by the formulas
\begin{align*}
b_{11} = {}& 2\lambda (2u_{0}^2 - u_{12} ) - 4\ii u_{0}^3  + 6\ii u_{0} u_{12}
  + 4u_{0} u_{0;x} + \ii u_{0;xx}  - u_{12;x} + 4\ii u_{2}u_{1} , \\
b_{12} = {} & 2\lambda^2( - 2u_{0}^2 + u_{12} ) + 2\lambda ( - 4u_{0}
u_{0;x} + u_{12;x}) - 4u_{0}^4 - 2u_{0}^2u_{12}  - 4u_{0} u_{0;xx} +
2u_{12}^2 \\
{}&{}- 4u_{0;x}^2 + u_{12;xx} - u_{2}u_{2;x} - 8u_{2}u_{1}u_{0}  -
u_{1}u_{1;x}, \\
b_{13} = {} & \lambda (u_{2;x} + \ii u_{1;x}  - 4\ii u_{2}u_{0}   +
4u_{1}u_{0} ) - 5\ii u_{0} u_{2;x}  + 5u_{0} u_{1;x} + u_{2;xx} +
\ii u_{1;xx}  \\
{}&{}- 8u_{2}u_{0}^2 + 2u_{2}u_{12}  - 4\ii u_{2}u_{0;x}  - 8\ii u_{1}u_{0}^2  + 2\ii u_{1}u_{12}  + 4u_{1}u_{0;x},\\
b_{21} = {} & 2(2u_0^2 - u_{12}), \\
b_{22} = {} & 2\lambda ( - 2u_{0}^2 + u_{12} ) - 4\ii u_{0}^3  + 6\ii u_{0} u_{12}   - 4u_{0} u_{0;x} + \ii u_{0;xx}  + u_{12;x} + 4\ii u_{2}u_{1} , \\
b_{23} = {} & u_{2;x} + \ii u_{1;x}  - 4\ii u_{2}u_{0}  + 4u_{1}u_{0} ,\\
b_{31} = {} & u_{2;x} - \ii u_{1;x}  + 4\ii u_{2}u_{0}   + 4u_{1}u_{0} ,\\
b_{32} = {} & \lambda ( - u_{2;x} + \ii u_{1;x}  - 4\ii u_{2}u_{0}   -
4u_{1}u_{0} ) - 5\ii u_{0} u_{2;x}  - 5u_{0} u_{1;x} - u_{2;xx} +
\ii u_{1;xx}  + 8u_{2}u_{0}^2\\
{}&{} - 2u_{2}u_{12}  - 4\ii u_{2}u_{0;x}  - 8\ii u_{1}u_{0}^2  + 2\ii u_{1}u_{12}   - 4u_{1}u_{0;x},\\
b_{33} = {} & 2 \ii ( - 4u_{0}^3 + 6u_{0} u_{12}  + u_{0;xx} + 4u_{2}u_{1}).
\end{align*}

The $\gsl(2|1)$-\/matrix 
\[
Q = \begin{pmatrix}
  0 & 1 & 0 \\
  0 & 0 & 0 \\
  0 & 0 & 0
\end{pmatrix}
\]
satisfies the equations
\[
\frac{\partial}{\partial \lambda} A = \bar{D}_x(Q) - \bbl A, Q \bbr, \qquad
\frac{\partial}{\partial \lambda} B = \bar{D}_t(Q) - \bbl B, Q \bbr.
\]
Solving the Cauchy problem
\[
  \frac{\partial}{\partial \lambda} S  = Q S, \quad  S|_{\lambda =0} = \boldsymbol{1},
\]
we obtain the $SL(2|1)$-\/matrix 
\[
S = \begin{pmatrix}
  1 & \lambda & 0 \\
  0 & 1 & 0 \\
  0 & 0 & 1
\end{pmatrix}.
\]
This matrix~$S$ defines the --\,obviously, smooth with respect to~$\lambda$\,--
family of gauge transformations that remove the
parameter~$\lambda$ from the zero\/-\/curvature representation~$\beta$, 
i.e.\ $(\beta)^{S^{-1}} = \beta|_{\lambda = 0}$. 
Consequently, the parameter $\lambda$ in $\beta$ is removable.
\end{example}

\begin{rem}%
  Marvan's computational and moreover, cohomological techniques from~\cite{Marvan2002,Marvan2010} seem to be
  working really fine in the $\BBZ_2$-\/graded set\/-\/up.\footnote{The \emph{horizontal cohomology} groups
    introduced by Marvan in~\cite{Marvan2002,Marvan2010} are informative for the algebraic approach to
    kinematic integrability, yet they may be hard to compute (in fact, this has not been attempted
    industrially). It is the removability of ``fake'' parameters in the zero\/-\/curvature representations
    which must be focused on first; whenever it is established that a parameter cannot be removed in a smooth
    way from a smooth family, the integration of PDE under study by using the inverse
    scattering~\cite{Faddeev,ZSh} should be attempted as the proper next step (or a nontrivial Gardner
    deformation of that system be derived from the family of zero\/-\/curvature representations, and the
    integrals of motion be constructed).}  However, let us say a word of caution.
\end{rem}

\begin{lemma}\label{lemZ2CoHom}
Let $\alpha = \alpha^{\bzero} + \alpha^{\bone}$ 
be a $\fg$-\/valued zero\/-\/curvature representation of 
a given $\BBZ_2$-graded equation $\cE$ such that
$\parity(\alpha^{\bzero})=\bzero$ and $\parity(\alpha^{\bone}) = \bone$.
Then Marvan's operator $\bar{\ddb}_{\alpha}=\bar{\Id}_h-\bbl\alpha,\cdot\bbr$
is not necessarily a differential.
\end{lemma}

We refer to equation~\eqref{EqMMDiff} below and to the papers~\cite{Marvan2002,Norway} for more details on the nature and use of the mapping~$\bar{\ddb}_{\alpha}$.

\begin{proof}
Let $\beta \in \fg\otimes\bar\Lambda^0(\cEinf)$ so that
$\beta = \bz + \bo$ and consider $\alpha =
\az + \ao$, where $\parity(\az) = \parity(\bz) = \bzero$ and
$\parity(\ao)=\parity(\bo)=\bone$. Then we have that
\begin{align*}
{}&\bar\ddb_{\alpha}\circ\bar\ddb_{\alpha}(\beta) =
\bar\ddb_{\alpha}( \bar{\Id}_h\beta - \bbl \alpha, \beta \bbr) = 
\bar{\Id}_h\circ\bar{\Id}_h \beta - \bar{\Id}_h(\bbl \alpha, \beta \bbr) -
  \bbl \alpha, \bar{\Id}_h\beta - \bbl \alpha, \beta \bbr\, \bbr 
\\
{}&\ {}= - \bbl \bar{\Id}_h\alpha, \beta \bbr + \bbl \alpha, \bar{\Id}_h \beta
  \bbr - \bbl \alpha , \bar{\Id}_h \beta \bbr + \bbl \alpha, \bbl \alpha, \beta \bbr \, \bbr
= \bbl \alpha, \bbl \alpha, \beta \bbr \, \bbr - \tfrac12 \bbl \bbl \alpha,
  \alpha \bbr, \beta \bbr 
\\
{}&\ {}= \bbl \alpha, \az\bz - \bz\az + \az\bo - \bo\az + \ao\bz - \bz\ao 
  + \ao\bo + \bo\ao \bbr
  - \bbl \az\az + \az\ao + \ao\az, \beta \bbr 
\\
{}&\ {}= \az\az\bz + \az\bz\az - \az\bz\az - \bz\az\az 
  + \az\az\bo + \az\bo\az - \az\bo\az - \bo\az\az 
\\
{}&\quad{}+ \az\ao\bz + \ao\bz\az - \az\bz\ao - \bz\ao\az
  + \az\ao\bo + \ao\bo\az
  + \ao\az\bz + \az\bz\ao
\\
{}&\quad{}- \ao\bz\az - \bz\az\ao + \ao\az\bo - \az\bo\ao
  - \ao\bo\az - \bo\az\ao + \ao\ao\bz - \ao\bz\ao 
\\
{}&\quad{}- \ao\bz\ao + \bz\ao\ao + \ao\ao\bo + \ao\bo\ao
  - \az\az\bz + \bz\az\az - \az\az\bo + \bo\az\az 
\\
{}&\quad{}- \ao\az\bz + \bz\ao\az - \ao\az\bo - \bo\ao\az
  - \az\ao\bz + \bz\az\ao - \az\ao\bo - \bo\az\ao 
\\
{}&\ {}= - \az\bo\ao - 2\bo\az\ao + \ao\ao\bz - 2\ao\bz\ao + \bz\ao\ao  
+ \ao\ao\bo + \ao\bo\ao - \bo\ao\az \neq 0.
\end{align*}
This argument shows that only for \emph{parity\/-\/even} zero\/-\/curvature representations (which are constrained
by $\alpha^{\bone} = 0$) is the operator $\bar{\ddb}_{\alpha}$ always 
a differential and does Marvan's horizontal cohomology
interpretation
~\cite{Marvan2002} work %
in the $\BBZ_2$-\/graded set\/-\/up.
\end{proof}

\begin{rem}%
We have not seen any nontrivial example of zero\/-\/curvature representations with nonzero odd part (i.e.\ such that $\alpha^{\bone} \neq 0$). 
It would be interesting to either find such example or prove that it cannot exist.
\end{rem}

Let us remember that Proposition~\ref{PropNonRemGrad} and Lemma~\ref{lemZ2CoHom} can be used to construct parametric families~$\alpha_{\lambda}$ of zero\/-\/curvature representations for an equation~$\cE$. In~\cite{Marvan2010}, the horizontal gauge cohomology complex~$\bar{H}^q_{\alpha_\lambda}(\cE, \fg)$ was associated with every such family. It is standard that the first horizontal gauge cohomology group~$\bar{H}^1_{\alpha_\lambda}(\cE, \fg)$ contains the obstructions to removability of a %
parameter~$\lambda$ (cf.\ section~\ref{GradedFN} above and~\cite{KKIgonin2003,JKIgonin2002}).

\begin{example}%
Using the technique described in~\cite{Marvan2010}, %
let us examine an $\gsl(2|1)$-\/valued zero\/-\/cur\-va\-tu\-re
representation for the equation %
which found by Tian and Liu (Case~F in~\cite{TianLiu5ord}, see also~\cite{TianWang2016}):
\begin{subequations}\label{eq5ordF}
\begin{align}
u_t = {} & {} u_{5x} + 2a uu_{xxx} + 4au_xu_{xx} + \tfrac{6a^2}{5}u^2u_x - a\xi_{xxx}\xi_x 
 + \tfrac{3a^2}{5}u\xi_{xx}\xi + \tfrac{3a^2}{5}u_x\xi_x\xi,\label{eq5ordFBosonic}\\
\xi_t={} & {} \xi_{5x} + 2au\xi_{xxx} + 3au_x\xi_{xx} + au_{xx}\xi_x + \tfrac{3a^2}{5}u^2\xi_x 
 + \tfrac{3a^2}{5}uu_x\xi, \qquad (a=5);
\end{align}
\end{subequations}
the parities are $\parity(u) = \bzero$ and $\parity(\xi) = \bone$. 

We start from the non\/--\/parametric zero\/--\/curvature representation $\alpha^{\text{5ord}}_0 = A_0\,\Id x + B_0\,\Id t$ for
system~\eqref{eq5ordF}:
\begin{align*}
A_0 = {} & {} \begin{pmatrix}
  0 & - u & \xi \\
  1 &   0 & 0   \\
  0 &-\xi & 0 
\end{pmatrix}, \\ 
B_0 = {} & {} \begin{pmatrix}
  6uu_x + u_{xxx} + 3\xi\xi_{xxx} &
    b_{12} &
      b_{13} \\
  6u^2 + 2u_{xx} + 6\xi \xi_{x} &
    - 6u u_{x} - u_{xxx} - 3\xi \xi_{xx} & 
      b_{23} \\
  9u \xi_{x} + \xi_{xxx} - 3\xi u_{x} &
    - 9u \xi_{xx} - \xi_{4x} - 6\xi_{x}u_{x} - 6\xi u^2 + \xi u_{xx} &
      0
\end{pmatrix},
\end{align*}
where
\begin{align*}
b_{12} = {} & {}   - 6u^3 - 8u u_{xx} - u_{4x} - 6u_{x}^2 + 3\xi_{xx}\xi_{x} + 3\xi u \xi_{x} - 2\xi \xi_{xxx},\\
b_{13} = {} & {} 9u \xi_{xx} + \xi_{4x} + 6\xi_{x}u_{x} + 6\xi u^2 - \xi u_{xx},\\
b_{23} = {} & {} 9u \xi_{x} + \xi_{xxx} - 3\xi u_{x}.
\end{align*}
This $\gsl(2|1)$-\/valued zero\/-\/curvature representation can easily be obtained by 
using %
two ideas: (\textit{i})~it should reduce to the standard %
$\gsl(2)$-\/valued zero\/-\/curvature representation of the Korteweg\/--\/de
Vries equation under a reduction of~\eqref{eq5ordFBosonic} to its higher symmetry;
(\textit{ii}) the sought\/-\/for structure is expected to be homogeneous with respect to the tuple of scaling weights
(see~\cite{HKKW,SsTools} in this context).

Let us recall from~\cite{Marvan2010} that for a given $\fg$-\/valued zero\/-\/curvature representation~$\alpha^{\text{5ord}}_0 =
A_0\,\Id x + B_0\,\Id t$, the $\fg$-\/valued $1$-\/form $A_1\,\Id x + B_1\,\Id t$ is a $\bar{\ddb}_{\alpha}$-\/cocycle
if and only if the block matrices
\[
  A^{[1]} = \begin{pmatrix}
    A_0 & 0 \\ A_1 & A_0
  \end{pmatrix},\qquad
  B^{[1]} = \begin{pmatrix}
    B_0 & 0 \\ B_1 & B_0
  \end{pmatrix}
\]
constitute a zero\/-\/curvature representation, i.e.
\[
  \bar{D}_t A^{[1]} - \bar{D}_x B^{[1]} + [A^{[1]}, B^{[1]}] = 0.
\]
Moreover, two cocyles differ by a coboundary if and only if the 
respective %
zero\/-\/curvature
representations are gauge equivalent with respect to a gauge matrix of the form
\[
  S^{[1]} = \begin{pmatrix}
    E & 0 \\ S & E
  \end{pmatrix},  
\]
with the unit matrix~$E$ at the diagonal and some matrix~$S$.

Using this condition and the analytic software \texttt{SsTools}~\cite{SsTools}, 
we found the cocycle~$A_1\,\Id x + B_1\,\Id t$, where
\begin{subequations}\label{eq5ordtrivialcocyle}
\begin{align}
  A_1 = {} & {} \begin{pmatrix}
    1 & 0 & 0 \\
    0 &-1 & 0 \\
    0 & 0 & 0
  \end{pmatrix}, \\
  B_1 = {} & {} \begin{pmatrix}
    6u^2 + 2u_{xx} + 6\xi \xi_{x} & - 12u u_{x} - 2u_{xxx} - 6\xi \xi_{xx} & 9u \xi_{x} + \xi_{xxx} - 3\xi u_{x}\\
    0 & - 6u^2 - 2u_{xx} - 6\xi \xi_{x} & 0 \\
    0 &  - 9u \xi_{x} - \xi_{xxx} + 3\xi u_{x} & 0
  \end{pmatrix},
\end{align}
\end{subequations}
for the non\/-\/parametric $\gsl(2|1)$-\/valued zero--\/curvature representation
$\alpha^{\text{5ord}}_0$ of system~\eqref{eq5ordF}.
Thus we extend the zero\/-\/curvature representation $\alpha^{\text{5ord}}_0$
to the $\gsl(2|1)$-\/valued
zero\/-\/curvature representation~$\alpha^{\text{5ord}} = A\,\Id x + B\,\Id t$ with 
a %
parameter~$\lambda$, here
\[
  A = \begin{pmatrix}             
    \lambda & - u  - \lambda^2 & \xi \\
    1 & - \lambda & 0 \\
    0 & - \xi &0 
  \end{pmatrix}, %
  \qquad
  B = \begin{pmatrix}
    b_{11} & b_{12} & b_{13} \\
    b_{21} & -b_{11} & b_{23} \\
    b_{31} & b_{32} & 0
  \end{pmatrix},
\]
where
\begin{align*}
b_{11} = {} & {} 2\lambda (3u^2 + u_{xx} + 3\xi \xi_{x}) + 6u u_{x} + u_{xxx} + 3\xi \xi_{xx}, \\
b_{12} = {} & {} 2\lambda^2( - 3u^2 - u_{xx} - 3\xi \xi_{x}) + 2\lambda ( - 6u u_{x} - u_{xxx} 
              - 3\xi \xi_{xx}) - 6u^3 - 8u u_{xx} - u_{4x} - 6u_{x}^2 \\
         {} & {} + 3\xi_{xx}\xi_{x} + 3\xi u \xi_{x} - 2\xi \xi_{xxx},\\
b_{13} = {} & {} \lambda (9u \xi_{x} + \xi_{xxx} - 3\xi u_{x}) + 9u \xi_{xx} + \xi_{4x} + 6\xi_{x}u_{x} + 6\xi u^2 - \xi u_{xx},\\
b_{21} = {} & {} 2(3u^2 + u_{xx} + 3\xi \xi_{x}),\\
b_{23} = {} & {} 9u \xi_{x} + \xi_{xxx} - 3\xi u_{x}, \\
b_{31} = {} & {} 9u \xi_{x} + \xi_{xxx} - 3\xi u_{x}, \\
b_{32} = {} & {} \lambda ( - 9u \xi_{x} - \xi_{xxx} + 3\xi u_{x}) - 9u \xi_{xx} - \xi_{4x} - 6\xi_{x}u_{x}
              - 6\xi u^2 + \xi u_{xx}.
\end{align*}
But is this parameter $\lambda$ (non)\/removable\,?

It is easy to see that the gauge transformation~$S^{[1]}$ given by
\[
  S^{[1]} = \begin{pmatrix} E  & 0 \\ S & E  \end{pmatrix},\quad
  \text{where}\quad S = \begin{pmatrix}
    0 & 1 & 0 \\ 0 & 0 & 0 \\ 0 & 0 & 0
  \end{pmatrix}.
\]
trivialises %
cocycle~\eqref{eq5ordtrivialcocyle}\,! Namely, we have that
\begin{align*}
  \left(A^{[1]}\right)^{S^{[1]}} = {} \begin{pmatrix} A_0 & 0 \\ 0 & A_0 \end{pmatrix},\qquad
  \left(B^{[1]}\right)^{S^{[1]}} = {} \begin{pmatrix} B_0 & 0 \\ 0 & B_0 \end{pmatrix}.
\end{align*}
Therefore, %
cocycle~\eqref{eq5ordtrivialcocyle} cannot be used to insert a nonremovable parameter in the zero\/-\/curvature representation~$\alpha^{\text{5ord}}_0$.

We deduce that the above family of zero\/-\/curvature representations $\alpha^{\text{5ord}}$ does not manifest that equation~\eqref{eq5ordF} 
indeed is integrable.\footnote{%
At the same time, we do not claim that no non\/-\/removable parameter can be inserted into the $\gsl(2|1)$-\/valued zero\/-\/curvature representation $\alpha^{\text{5ord}}_0$. Indeed, to establish that one must prove the vanishing of the respective gauge cohomology group.}
This is why in Example~\ref{Ex98} on p.~\pageref{Ex98} below
we shall consider much larger, $\mathfrak{sl}(9|8)$-\/valued zero\/-\/curvature representations for~\eqref{eq5ordF}. We argue on p.~\pageref{pClaim} that the parameter~$\lambda$ in that family is not removable using gauge transformations. %
\end{example}

\section{Families of coverings and the Fr\"olicher\/--\/Nijenhuis bracket
}\label{GradedFN}
\noindent%
In this section we recall another mechanism for deforming %
nonlocal structures over partial differential (super-)\/equations~$\cE$.
Referring to the notion of Fr\"o\-li\-cher\/--\/Nijenhuis bracket,
this concept is better known in the context of B\"acklund (auto)\/trans\-for\-ma\-ti\-ons between PDEs, see~\cite{JKIgonin2002,AVKVestnikMGU2002,KiselevGolovkoActaAM2004}.
The two deformation strategies from the preceding and present sections will be brought together in what follows.

\subsection{The structure element of a covering}
Consider a $(k_0|k_1)$-\/dimensional 
covering $\tau\colon\tilde{\cE} = W \times \cEinf \to \cEinf$ with even nonlocal coordinates
$w^1,\dots,w^{k_0}$ and odd nonlocal coordinates $f^1,\dots,f^{k_1}$ on the $(k_0|k_1)$-\/dimensional %
fibre superspace%
~$W$ (see Definition~\ref{defCovering} on p.~\pageref{defCovering}%
). The prolongations $\tilde{D}_{x^i}$ of the total derivatives $\bar{D}_{x^i}$ to the
covering equation~$\tilde{\cE}$ are given by the formulas~\cite{BVV,KK2000}
\[
\tilde{D}_{x^i} = \bar{D}_{x^i} + w^p_{x^i} \frac{\partial}{\partial w^p} +
f^q_{x^i} \frac{\vec{\partial}}{\partial f^q},
\qquad 1\leqslant i\leqslant n.
\]
These total derivatives $\tilde{D}_{x^i}$ determine the Cartan
distribution $\cC(\tilde{\cE})$ on the covering equation~$\tilde{\cE}$.
In turn, the Cartan distribution $\cC(\tilde{\cE})$ yields the Cartan
connection $\cC_{\tilde{\cE}}\colon \Gamma (TM) \to \Gamma (T\tilde{\cE})$
such that $\cC_{\tilde{\cE}}\colon \dd/\dd x^i \mapsto \tilde{D}_{x^i}$.
Using those differential one\/-\/forms from $\Lambda^1(\tilde{\cE})$ which annihilate the horizontal $n$-\/dimensional planes of the Cartan distribution~$\cC(\tilde{\cE})$, 
we obtain the linear $\Gamma (T\tilde{\cE})$-\/valued connection form
$U_{\tilde{\cE}} \in \Gamma T(\Lambda^1(\tilde{\cE}))$, also called the \emph{structure element} of the covering~$\tau$ in this setting.

Specifically, for $\cA\mathrel{{:}{=}}C^\infty(M)$, for the properly understood inductive limit
${\cB\mathrel{{:}{=}}C^\infty(\tilde{\cE})}$ of algebras filtered by the jet orders, and embedding homomorphism
$\varphi=(\pi\circ\tau)^*\colon \cA\hookrightarrow\cB$, it is readily seen that the composition
$\Id_{\text{dR}}\circ\varphi\colon \cA\to\Lambda^1(\cB)$ is a derivation. Consequently, we can consider the
derivation
$\cC_{\tilde{\cE}}\bigl(\Id_{\text{dR}}\circ\varphi \bigr)\in\Gamma T\bigl(\Lambda^1(\tilde{\cE})\bigr)$.  The
structure element $U_{\tilde{\cE}} \in \Gamma T(\Lambda^1(\tilde{\cE}))$ of the covering~$\tau$ is defined by
the formula $U_{\tilde{\cE}} = -\cC_{\tilde{\cE}}\bigl(\Id_{\text{dR}}\circ\varphi\bigr) + \Id_{\text{dR}}$
(see~\cite{KKIgonin2003,JKIgonin2002,JKFlatCon} and references therein).
In coordinates, we have that\footnote{The notion of Cartan differential $\Id_{\cC}$ and its restriction to equations~$\cE^\infty$ is recalled on p.~\pageref{pCartanDiff}.}
\begin{equation}\label{EqConForm}
U_{\tilde{\cE}} = 
  \bar{\Id}_{\cC}(u^k_{\sigma_\bzero})\frac{\partial}{\partial  u^k_{\sigma_\bzero}} 
  + \bar{\Id}_{\cC}(\xi^a_{\sigma_\bone})\frac{\vec{\partial}}{\partial  \xi^{a}_{\sigma_\bone}} 
  + (\Id w^p - w^p_{x^i}\, \Id x^i) \frac{\partial}{\partial w^p} 
  + (\Id f^q - f^q_{x^i}\, \Id x^i) \frac{\vec{\partial}}{\partial f^q}.
\end{equation}
The Cartan connection on~$\tilde{\cE}$ is flat:
\[
\fnl U_{\tilde{\cE}}, U_{\tilde{\cE}} \fnr = 2\times\text{curvature}=0;
\]
we recall that the Fr\"olicher\/--\/Nijenhuis bracket $\fnl \cdot, \cdot
\fnr$ on the space $\Gamma T(\Lambda^*(\tilde{\cE}))$ of vector\/-\/valued differential forms
is defined by the formula~\cite{KK2000}%
\[
\fnl \Omega, \Theta \fnr (g) = \Lie_{\Omega}(\Theta(g)) - (-1)^{rs
  + \parity(\Omega)\parity(\Theta)} \Lie_{\Theta}(\Omega(g))
\]
for $\Omega \in \Gamma T(\Lambda^r(\tilde{\cE}))$, $\Theta \in
\Gamma T(\Lambda^s(\tilde{\cE}))$, and $g \in C^{\infty}(\tilde{\cE})$.
Here $\Lie_{\Omega} = \mathrm{i}_{\Omega} \circ \Id + \Id \circ \mathrm{i}_{\Omega}$ is the Lie derivative.

Let $\tau_\lambda \colon \tilde{\cE}_\lambda = W_\lambda\times \cEinf \to \cEinf$ be a smooth family of
coverings over $\cEinf$ depending on a parameter~$\lambda\in\BBC$ and
$U_{\lambda}$ be the corresponding %
structure element of~$\tau_\lambda$. 
In agreement with%
~\cite{JKIgonin2002}, we assume that the distributions
$\cC(\tilde{\cE_\lambda})$ are diffeomorphic to each other 
at different values of~$\lambda$ %
under a smooth family of diffeomorphims of the manifolds~$\tilde{\cE_\lambda}$.
The evolution of $U_{\lambda}$ with respect to $\lambda$ is
described by the equation~\cite{KKIgonin2003,JKIgonin2002}%
\begin{equation}\label{eqFNXshadow}
\frac{\Id}{\Id \lambda} U_{\lambda} = \fnl X, U_{\lambda} \fnr,
\end{equation}
where $X\in \Gamma (T\tilde{\cE})$ is some vector field on~$\tilde{\cE}_\lambda$.

\subsection{Two realisations of Lie superalgebras}
In the covering~$\tau$ which has been considered so far, the superdimension of the fibre~$W$ is $(k_0|k_1)$; the covering is realised in terms of vector fields on the fibres. To narrow the class of vector field subalgebras at hand --\,in particular, to force such vector fields belong to a given Lie algebra~$\gothg$ that appeared in section~\ref{GradedMarvan} in the construction of $\gothg$-\/valued zero\/-\/curvature representations\,-- let us recall two techniques of matrix vs (non)linear vector field realisations of Lie algebras.
We have that elements~$\rho(g)$ of matrix representation for~$g\in\gothg$ act by endomorphisms on a vector superspace of superdimension $(k_0+1|k_1)$. The excess of parity\/-\/even dimension allows us to view the respective graded Cartesian coordinates as homogeneous coordinates on projective superspaces~$W$ such that vector fields $\boldsymbol{\varrho}(g) \in \Gamma(TW)$ on them belong to the other representation, $\boldsymbol{\varrho} \colon \fg \to \Vect(W%
;\text{poly})$, of the Lie algebra~$\gothg$.

Specifically, let $\fg\subseteq \gl(k_0+1|k_1)$ be a finite\/-\/dimensional Lie
superalgebra with basis~$e_i$, here $k_0$,\ $k_1 \geqslant 0$
and the index~$i$ runs from~$1$ to the dimension of~$\fg$. 
We consider two representation of~$\fg$:
\begin{enumerate}
 \item $\rho \colon \fg \to \Mat(k_0+1,k_1)$, that is,
a matrix representation;%
 \item $\boldsymbol{\varrho} \colon \fg \to \Vect(W%
;\text{poly})$, which is 
the representation in the space of vector fields with polynomial
coefficients on the $(k_0|k_1)$-\/dimensional supermanifold~$W$ with
local parity\/-\/even coordinates $w^1$,\ $\dots$,\ $w^{k_0}$
and $f^{1}$,\ $\dots$,\ $f^{k_1}$ of odd parity. 
\end{enumerate}
Let us recall an explicit %
construction %
of such representations~$\boldsymbol{\varrho}$; it %
will be essential in what follows.

\begin{example}[Nonlinear realisations %
of Lie algebras in the spaces of vector fields via the
projective substitution~\cite{RoelofsThesis}]\label{exProjSub}%

Let $N$ be a $(k_0+1| k_1)$-\/dimensional manifold.  Because the reasoning is local, consider a
chart~$\mathcal{U}\subseteq N$ equipped with a $(k_0+k_1+1)$-\/tuple of rectifying
coordinates %
$\boldsymbol{v} = (v^0$,\ $\dots$,\ $v^{k_0+k_1} )$, %
where $v^{0}$,\ $\dots$,\ $v^{k_0}$ are even coordinates and $v^{k_0+1}$,\ $\dots$,\ $v^{k_0+k_1}$ are odd coordinates.
By definition, put
\[
\partial_{\boldsymbol{v}} = (\smash{\vec{\partial}_{v^0}}, %
\dots, \smash{\vec{\partial}_{v^{k_0+k_1}}} 
)^{\mathrm{t}}.  %
\]
For the matrix Lie superalgebra %
$\fg\subseteq\gl(k_0+1| k_1)$ at hand, take any matrix $%
{g}\in\fg$
and represent it %
in the space of \emph{linear} vector fields on the domain 
in~$\BBR^{k_0+1| k_1}$ by using the formula
\[
{g}\longmapsto V_{%
{g}} = \boldsymbol{v}g\partial_{\boldsymbol{v}}.
\]
By construction, the linear vector field representation $%
{g}\mapsto V_{%
{g}}$ of the matrix Lie
algebra~$\fg$ preserves all the 
commutation relations in it, %
\[
[%
V_{%
{g}}, V_{%
{h}} ]%
= [%
\boldsymbol{v}g\partial_{\boldsymbol{v}},
\boldsymbol{v}h\partial_{\boldsymbol{v}} ]%
= \boldsymbol{v}\bbl
g,h \bbr 
\partial_{\boldsymbol{v}} =
 V_{\bbl 
{g}, %
{h} \bbr
}, \qquad %
\forall\:{h}, %
{g}\in\fg.
\]
The problem we are solving is the realisation of matrix Lie superalgebra~$\fg$ by using vector fields with
(non)\/linear coefficients.
Consider a point $\boldsymbol{x}\in\BBR^{k_0+1|k_1}$ --\,originally, from the chart $\mathcal{U}\subseteq N$\,-- with nonzero coordinate $v^0(\boldsymbol{x})\mathrel{{=}{:}}\mu\neq0$.
By construction, $\boldsymbol{v} = (v^0$,\ $\dots$,\ $v^{k_0+k_1} )$~is the tuple of Cartesian coordinates on the image of~$\mathcal{U}$ under the coordinate mapping. 
Consider the locally defined mapping $p\colon\boldsymbol{v}(\mathcal{U})\subseteq
\BBR^{k_0+1| k_1}\to\BBR^{k_0| k_1}$ that takes every point
$\boldsymbol{v}=(v^0$,\ $\ldots$,\ $v^{k_0+k_1})$ from the domain at hand 
to the point %
$(w^1,\ \dots,\ w^{k_0+k_1} )\in\BBR^{k_0| k_1}$, where
\[
w^i = \frac{\mu v^i}{v^0},\quad 1\leqslant i \leqslant k_0+k_1.
\]
The differential $\Id p_{\bv} \colon T_{\bv} \BBR^{k_0+1| k_1}\to T_{p(\bv)}\BBR^{k_0| k_1}$ at the point
$\bv\in\BBR^{k_0+1| k_1}$ acts on the basic vectors from the $(k_0+k_1+1)$-\/tuple~$\dd_{\boldsymbol{v}}$ as follows,
\begin{align*}
\Id p_\bv\left( \frac{\dd}{\dd v^0}\right)
  &= \sum_{j=1}^{k_0+k_1} \frac{\dd w^j}{\dd v^0}\,\frac{\dd}{\dd w^j} 
  = \sum_{j=1}^{k_0+k_1}-\frac{\mu v^j}{(v^0)^2}\,\frac{\dd}{\dd w^j}, \\
\Id p_\bv\left( \frac{\dd}{\dd v^i}\right)
  &= \sum_{j=1}^{k_0+k_1}\frac{\dd w^j}{\dd v^i}\,\frac{\dd}{\dd w^j} = \frac{\mu}{v^0}\,\frac{\dd}{\dd w^i},
  \qquad \lefteqn{1\leqslant i \leqslant k_0+k_1.}
\end{align*}
By definition, put
\[
\bw = (\mu, w^1, \ldots, w^{k_0+k_1}), \qquad 
\dd_{\bw} = \left( -\frac{1}{\mu}\sum_{j=1}^{k_0+k_1} w^j \frac{\vec{\dd}}{\dd w^j}, 
\ %
\frac{\vec{\dd}}{\dd w^1}, \ldots,
  \frac{\vec{\dd}}{\dd w^{k_0+k_1}} \right)^{\mathrm{t}}.
\]
Now it is readily seen that %
the vector field $X_{%
{g}}=\mathrm{d}p\,(V_{%
{g}})$
is expressed by the formula %
\begin{equation}
\label{wroelofs}
X_{%
{g}} = \boldsymbol{w} g \partial_{\boldsymbol{w}}.
\end{equation}
Generally speaking, the vector field~$X_{%
{g}}$ on the respective subset of %
the target space $\BBR^{k_0| k_1}$ is nonlinear with respect to the variables~$w^0$,\ $\ldots$,\ %
$w^{k_0+k_1}$. %
Nevertheless, the commutation relations between vector fields of such type
are inherited 
from the relations in Lie algebra~$\fg\ni %
{g},%
{h}$:
\[
[X_{%
{g}}, X_{%
{h}} ]%
= [%
\mathrm{d}p\,(V_{%
{g}}),\mathrm{d}p\,(V_{%
{h}}) ]%
= \mathrm{d}p\,([%
V_{%
{g}}, V_{%
{h}} ]%
) = \mathrm{d}p\,(V_{\bbl 
{g}, %
{h} \bbr
}) =
  X_{\bbl 
{g}, %
{h} \bbr
}.
\]
Take $X_%
{g}$ for the representation $\boldsymbol{\varrho}(%
{g})$ of elements~$%
{g}$ 
of Lie %
superalgebra~$\fg$. 
\end{example}

\begin{notation}
Whenever $\fg\subseteq\mathfrak{gl}(n_\bzero|n_\bone)$ is a finite\/-\/dimensional %
Lie superalgebra, it admits the %
tautological matrix representation~$\rho\colon\fg\to\mathfrak{gl}(n_\bzero|n_\bone)$ and %
realisation~$\boldsymbol{\varrho}$ in spaces of vector fields from Example~\ref{exProjSub}.
We denote by~$\nabla\colon\rho\rightleftarrows\boldsymbol{\varrho}$ the switch between these representations.
\end{notation}

We refer to~\cite{PopovychBoykoNesterenko2003,Shchepochkina2006} for other examples of realisations %
of Lie algebras by using vector fields.

\subsection{Examples of the structure element deformation}
Example~\ref{exProjSub} yields a regular procedure that takes Lie super\/-\/algebra $\fg$-\/valued zero\/-\/curvature representations to certain %
$(k_0|k_1)$-\/dimensional %
coverings over the %
$\BBZ_2$-\/graded PDE at hand;
the coefficients of Maurer\/--\/Cartan's horizontal one\/-\/form~$\alpha$ determine the rules to differentiate the nonlocal variables with respect to the independent coordinates such as~$x$ and~$t$. Let us see what the deformation mechanism of Fr\"olicher\/--\/Nijenhuis bracket can then do for such 
coverings~--- and let us inspect what the vector fields~$X$ in~\eqref{eqFNXshadow} mean in terms of the PDE under study.

\begin{example}\label{exThiune}
Consider the $N{=}2$, $a{=}4$ SKdV equation~\eqref{SKdVComponents} and the
family of coverings over it
derived from the zero\/-\/curvature representation which we addressed 
in Example~\ref{exN2ZCRDas}. 
Let us find the vector field~$X$ corresponding to that family.
We are %
interested in finding those solutions of~\eqref{eqFNXshadow} which do not degenerate under the reduction of~\eqref{SKdVComponents} to its bosonic limit. In turn, for that system of KdV\/-\/type we are interested in finding only those solutions of~\eqref{eqFNXshadow} which stem from the deformation generators~$X$ for the Korteweg\/--\/de Vries equation for~$u_{12}$; we recall that the KdV equation is contained in the bosonic limit of~\eqref{SKdVComponents} by Mathieu's construction.

This allows us to analyse the structural element's deformation problem ``inside\/-\/out'' in three steps: first, we do it for the Korteweg\/--\/de Vries equation; we proceed with the Kaup\/--\/Boussinesq hierarchy and finally, we recover the full $(2|2)$-\/dimensional supergeometry of the $N=2$ supersymmetric $a=4$-\/equation~\eqref{SKdVComponents}.

We begin with %
the Korteweg\/--\/de Vries equation
\begin{equation}\label{kdv}
    u_{12;t} = - u_{12;xxx} - 6u_{12}u_{12;x}.
\end{equation}
Over it, we consider the covering 
derived from the Gardner deformation~\cite{Miura68},
\begin{subequations}\label{eqCovGardnerKdV}
\begin{align}
      w_x& = \tfrac{1}{\veps}( w - u_{12}) - \veps w^2\label{eqCovGardnerKdVx},\\
      w_t& = \tfrac{1}{\veps}( u_{12;xx} + 2u_{12}^2) + \tfrac{1}{\veps^2} u_{12;x} + \tfrac{1}{\veps^3}u_{12}
      + \left(-2u_{12;x} - \tfrac{2}{\veps}u_{12} - \tfrac{1}{\veps^3}\right)w + \left(2\veps u_{12} +
        \tfrac{1}{\veps}\right) w^2. \label{eqCovGardnerKdVt}
\end{align}
\end{subequations}
The Cartan structure element for the covering~\eqref{eqCovGardnerKdV} is as follows,
\[
  U_{\tilde{\cE}} = \bar{\Id}_{\cC}(u_{12;\sigma})\frac{\dd}{\dd u_{12;\sigma}}
  + (\Id w - w_x\,\Id x - w_t\,\Id t)\frac{\dd}{\dd w},
\]
where $w_x$ and $w_t$ are given by~\eqref{eqCovGardnerKdV}.
The solutions of equation~\eqref{eqFNXshadow} for this case
are
  \begin{align*}
    X_1 &= \veps^{-2}(   - x\,{\dd}/{\dd x}
    - 3t\,{\dd}/{\dd t}
    + 2u_{12}  \,{\partial}/{\partial u_{12}}
    + \ldots
    + 2w  \,{\dd}/{\dd w} ),\\
  X_2 &= -2\veps( 6t\,{\dd}/{\dd x}
    + \,{\dd}/{\dd u_{12}}
    + \dots)
    - {\dd}/{\dd w}.
\end{align*}  

Let us show that the vector field~$X_1$
can be lifted to the corresponding covering over the higher symmetry, 
\begin{subequations}\label{kb3}
\begin{align}
u_{0;t}&=-u_{0;xxx}+\bigl(4 u_0^3  -6u_0u_{12}\bigr)_x,\\
u_{12;t}&=-u_{12;xxx}-6u_{12}u_{12;x} +12u_{0;x}u_{0;xx}+6u_0u_{0;xxx} +12\bigl(u_0^2u_{12}\bigr)_x,
\end{align}
\end{subequations}
of the Kaup\/--\/Boussinesq equation~\cite{Kaup75}
\[
u_{0;s} = (- u_{12} + 2u_0^2)_x, \qquad u_{12;s} = (u_{0;xx} + 4u_0u_{12})_x.
\]
We recall that system~\eqref{kb3} is the bosonic limit of~\eqref{SKdVComponents} with~$a{=}4$ under setting~$u_1=u_2=0$.

A family of coverings over equation~\eqref{kb3} is determined by the 
formulas\footnote{Here and in what follows we underline the covering that encodes Gardner's deformation~\eqref{eqCovGardnerKdV} for the classical KdV equation~\eqref{kdv}.}
\begin{subequations}\label{covKB}
\begin{align}
  \underline{w_x} = {}& \underline{- \veps w^2} + \veps^{-1} (\underline{w - u_{12}} - u^2_0) + \ii \veps^{-2}u_0,\\
  \underline{w_t} = {}& 2\veps w^2( - 2u_{0}^2 + \underline{u_{12}} ) + 2w ( - \ii w u_{0} + 4u_{0} u_{0;x}
  \underline{- u_{12;x}}) + \veps^{-1}(\underline{w^2 - 2w u_{12} +  2u_{12}^2} \notag\\
  {}&{} \underline{+ u_{12;xx}} + 2\ii w u_{0;x} - 4u_{0}^4 - 2u_{0}^2u_{12} - 4u_{0} u_{0;xx}
  + 4w u_{0}^2 - 4u_{0;x}^2 ) + \veps^{-2}(2\ii w u_{0}  \notag\\
  {}&{} + 2\ii u_{0}^3 - 4\ii u_{0} u_{12} - 4u_{0} u_{0;x} - \ii u_{0;xx} + \underline{u_{12;x}}) +
  \veps^{-3}(\underline{u_{12} -w } - u_{0}^2 - \ii u_{0;x}) - \ii \veps^{-4}u_{0}.
\end{align}
\end{subequations}
The Cartan structure element for the covering~\eqref{covKB} is as follows,
\[
  U_{\tilde{\cE}} = \bar{\Id}_{\cC}(u_{0;\sigma})\frac{\dd}{\dd u_{0;\sigma}} + \bar{\Id}_{\cC}(u_{12;\sigma})\frac{\dd}{\dd u_{12;\sigma}}
  + (\Id w - w_x\,\Id x - w_t\,\Id t)\frac{\dd}{\dd w},
\]
where the derivatives~$w_x$ and~$w_t$ are given by formulas~\eqref{covKB} and $\sigma$~is a summation index.
At every~$\veps\neq0$ such coverings are constructed %
by the change of Lie algebra realisation~$\nabla$ %
in the zero\/-\/curvature representation for~\eqref{kb3}. %
In turn, that representation\footnote{Remarkably, that zero\/-\/curvature representation for~\eqref{kb3} was
  re\/-\/discovered %
  in~\cite{BorisovPavlovZykov2001} not in the context of super\/-\/system~\eqref{SKdVComponents}.}  
is obtained by using the reduction $u_1=u_2=0$ in the zero\/-\/curvature representation $\alpha^{N{=}2}$ for the
$N{=}2$, $a{=}4$ SKdV equation~\eqref{SKdVComponents} (see \cite{Das,JMP2012} and Example~\ref{exN2ZCRDas}).

For this family of coverings over system~\eqref{kb3}, 
the solution of equation~\eqref{eqFNXshadow} is given by the vector field
\[
X = 
  \veps^{-1} (
  - x \,{\dd}/{\dd x}
  - 3 t\, {\partial}/{\partial t}
  +  u_0\, {\partial}/{\partial u_0}
  + 2 u_{12}\, {\partial}/{\partial u_{12}}
  + \ldots
  + 2 w\, {\partial}/{\partial w}).
\]
We note that, the same as it is in the case of KdV equation~\eqref{kdv},
we obtain the vector field corresponding to the scaling symmetry
(see Appendix~\ref{appParamNotes}; %
we refer to diagram~\eqref{CDSasaki} on p.~\pageref{CDSasaki} in particular).

Finally,\label{exThiuneN2} 
let us consider the full $N{=}2$, $a{=}4$ SKdV equation~\eqref{SKdVComponents} 
and over it, let us
construct a $(1|1)$-\/dimensional covering by switching to a different 
realisation %
of the Lie superalgebra in
the $\gsl(2|1)$-\/valued zero\/-\/curvature representation~$\alpha$, which was 
considered in Example~\ref{exN2ZCRDas} on p.~\pageref{exN2ZCRDas}.  
Using the %
realisation~$\boldsymbol{\varrho}$ %
from Example~\ref{exProjSub}, we obtain 
the $(1|1)$-\/dimensional covering over $N{=}2$, $a{=}4$ SKdV 
equation~\eqref{SKdVComponents}:
\begin{align*}
\underline{w_x} = {} & \underline{- \veps w^2} + (f_{2}u_{2} -
\ii f_{2}u_{1} ) + \veps^{-1} (\underline{w - u_{12}}  -
u_{0}^2  ) + \veps^{-2} \ii  u_{0}, \\
f_x = {} &  - \veps w f_{2} +  \ii   u_{0} f_{2} + \veps^{-1}(f_{2} + u_{2} +
 \ii u_{1}),\\
\underline{w_t} = {} & \veps ( - 4w^2u_{0}^2 + \underline{2w^2u_{12}}  + f_{2}w u_{2;x} - \ii f_{2}w u_{1;x}   + 4\ii f_{2}u_{2}w u_{0}   + 4f_{2}u_{1}w u_{0} ) - 2\ii w^2u_{0}  
\\
{}&{}+ 8w u_{0} u_{0;x} \underline{- 2w u_{12;x}} - 5\ii f_{2}u_{0} u_{2;x}   -
5f_{2}u_{0} u_{1;x} - f_{2}u_{2;xx} + \ii f_{2}u_{1;xx}   - f_{2}u_{2}w  \\
{}&{} + 8f_{2}u_{2}u_{0}^2 - 2f_{2}u_{2}u_{12}  - 4\ii f_{2}u_{2}u_{0;x}  
+ \ii f_{2}u_{1}w   -
8\ii f_{2}u_{1}u_{0}^2  + 2\ii f_{2}u_{1}u_{12}  - 4f_{2}u_{1}u_{0;x} \\
{}&{}+ \veps^{-1}(\underline{w^2 + 2u_{12}^2 - 2w u_{12} + u_{12;xx}} + 4w u_{0}^2   + 2\ii w u_{0;x}  -
4u_{0}^4 - 2u_{0}^2u_{12}  - 4u_{0} u_{0;xx}  \\
{}& {}    - 4u_{0;x}^2 - \ii f_{2}u_{2}u_{0}   - f_{2}u_{1}u_{0}  -
u_{2}u_{2;x} - 8u_{2}u_{1}u_{0}  - u_{1}u_{1;x}) + \veps^{-2}(2\ii w u_{0}
 + 2\ii u_{0}^3  \\
{}&{}- 4\ii u_{0} u_{12}   - 4u_{0} u_{0;x} - \ii u_{0;xx}  + \underline{u_{12;x}} -
2\ii u_{2}u_{1} ) + \veps^{-3}( \underline{ u_{12} - w}  - u_{0}^2  - \ii u_{0;x} ) -
\veps^{-4} \ii u_{0},
\\
f_t = {} & 2\veps w ( - 2f_{2}u_{0}^2 + f_{2}u_{12} ) + w u_{2;x} + \ii w
u_{1;x}  - 2 \ii f_{2}w u_{0}  + 4\ii f_{2}u_{0}^3  - 6\ii f_{2}u_{0} u_{12}  +
4f_{2}u_{0} u_{0;x} \\
{}&{}- \ii f_{2}u_{0;xx}  - f_{2}u_{12;x} - 4\ii f_{2}u_{2}u_{1}  - 4\ii u_{2}w
u_{0}  + 4u_{1}w u_{0}  + \veps^{-1}(5\ii u_{0} u_{2;x}  - 5u_{0} u_{1;x}
\\
{}&{}- u_{2;xx} - \ii u_{1;xx}  + f_{2}w  + 2f_{2}u_{0}^2 - f_{2}u_{12}  +
\ii f_{2}u_{0;x}  + u_{2}w  + 8u_{2}u_{0}^2 - 2u_{2}u_{12}  \\
{}&{}+ 4\ii u_{2}u_{0;x}  + \ii u_{1}w   + 8\ii u_{1}u_{0}^2  - 2\ii u_{1}u_{12}   -
4u_{1}u_{0;x}) + \veps^{-2}( - u_{2;x} - \ii u_{1;x}  + \ii f_{2}u_{0}  \\
{}&{}+ 3\ii u_{2}u_{0}   - 3u_{1}u_{0} ) - \veps^{-3}(f_{2} + u_{2} + \ii u_{1} ).
\end{align*}
The Cartan structure element is, therefore, %
\begin{align*}
  U_{\tilde{\cE}} = {}&{}\bar{\Id}_{\cC}(u_{0;\sigma})\frac{\dd}{\dd u_{0;\sigma}} 
  + \bar{\Id}_{\cC}(u_{1;\sigma})\frac{\dd}{\dd u_{1;\sigma}} + \bar{\Id}_{\cC}(u_{2;\sigma})\frac{\dd}{\dd u_{2;\sigma}}
  + \bar{\Id}_{\cC}(u_{12;\sigma})\frac{\dd}{\dd u_{12;\sigma}}\\
  {}&{}+ (\Id f - f_x\,\Id x - f_t\,\Id t)\frac{\dd}{\dd f}
  + (\Id w - w_x\,\Id x - w_t\,\Id t)\frac{\dd}{\dd w}.
\end{align*}
The reduction $u_1=u_2=0$ maps this covering over $N{=}2$, $a{=}4$ SKdV equation to the %
covering over higher
symmetry~\eqref{kb3} of the Kaup\/--\/Boussinesq equation, see above.
We lift the solution of equation~\eqref{eqFNXshadow} for the covering over higher symmetry~\eqref{kb3} of the
Kaup\/--\/Boussinesq equation to the covering over the $N{=}2$, $a{=}4$ SKdV equation.
That %
solution of equation~\eqref{eqFNXshadow} is the vector field 
\begin{multline*}
  X =   \veps^{-1}(
  - x \,{\dd}/{\dd x}
  - 3t \,{\dd}/{\dd t}
  + u_0 \,{\dd}/{\dd u_0}
  + \tfrac32 u_1 \,{\dd}/{\dd u_1}
  + \tfrac32 u_2 \,{\dd}/{\dd u_2} \\
  + 2u_{12} \,{\dd}/{\dd u_{12}}
  + \cdots 
  + 2 w \,{\dd}/{\dd w} + \tfrac32 f \,{\dd}/{\dd f}
  ).
\end{multline*}
It has been computed %
by solving equation~\eqref{eqFNXshadow}
explicitly %
with the help of analytic 
software%
~\cite{SsTools}.

We note again that --\,as we had it in the above two reductions of the $N{=}2$, $a{=}4$-SKdV\,-- we obtain the
vector field corresponding to the scaling symmetry of the underlying equation.%
\end{example}

\begin{rem}%
For scaling\/-\/invariant families of coverings depending on a parameter 
the homogeneity weight of which is not equal to zero,
one could always try to find a solution of equation~\eqref{eqFNXshadow} by properly extending %
the scaling symmetry of the underlying PDE,
cf.~\cite{JKIgonin2002,AVKVestnikMGU2002}.
\end{rem}

\section{Zero\/-\/curvature representations and coverings}\label{secZCRCovering}
\noindent
In this section we bring together the gauge geometry of zero\/-\/curvature representations from section~\ref{GradedMarvan} and the %
deformation of coverings as described in section~\ref{GradedFN}. 
In particular, we study the relation between the construction of parametric
families of zero\/-\/curvature representations and deformation of the corresponding %
nonlocalities %
by using the Fr\"o\-li\-cher\/--\/Nijenhuis bracket.
We let the geometry stay in the context of (super-)KdV type systems; therefore, we let $x^1=x$ and $x^2=t$ be the two independent variables.

Let $\alpha = a^i \rho(e_i)\,\mathrm{d}x + b^j \rho(e_j)\,\mathrm{d}t$ be 
a $\fg$-\/valued zero\/-\/curvature representation for the
system~$\cE$. Over its infinite prolongation~$\cEinf$,
construct a $(k_0|k_1)$-\/dimensional covering with 
the $k_0$ parity\/-\/even and $k_1$ parity\/-\/odd nonlocal fibre
variables~$w^\ell$ such that $\parity(w^\ell)=\bzero$ if $1\leqslant \ell\leqslant k_0$ and $\parity(w^\ell)=\bone$ if $\ell>k_0$ and such that
\begin{subequations} \label{covw} 
\begin{align}
 w^\ell_x & = - a^i \boldsymbol{\varrho}(e_i) \inner \Id w^\ell,\\
 w^\ell_t & = - b^j \boldsymbol{\varrho}(e_j) \inner \Id w^\ell, \quad \ell = 1\dots (k_0+k_1).
\end{align}
\end{subequations}
Consider two mappings: 
\begin{align}
\ddb_{\alpha} &= \bar{\Id}_h - 
\bbl \alpha, \cdot \bbr \colon
\fg\otimes\Lambda^{0}(\cEinf) \to \fg\otimes\Lambda^{1}(\cEinf),
\label{EqMMDiff}\\
\intertext{see~\cite{Marvan2002} and~\cite{Norway}, and} 
\partial_{U} &=
\fnl \cdot, U_{\lambda} \fnr \colon \Gamma T(\Lambda^{0}(\tilde{\cE})) \to \Gamma T(\Lambda^{1}(\tilde{\cE})),\notag
\end{align}
see~\cite{JKIgonin2002,JKFlatCon}.
We recall that the mappings $\ddb_{\alpha}$ and~$\partial_U$ yield the
horizontal~\cite{Marvan2002} and Cartan~\cite{JKIgonin2002} cohomologies, respectively. 
However, we claim that in the geometry at hand one of these two differentials is 
a particular instance of the other by virtue of the
switch $\rho\rightleftarrows\boldsymbol{\varrho}$ between the Lie superalgebra %
representations.

\begin{lemma}\label{lemmaCommute}
The following diagram is commutative\textup{:}
\[
\begin{CD}
  \fg\otimes\Lambda^{0}(\cEinf)
  @>\rho>>
  \Mat(k_0+1| k_1)\otimes\Lambda^{0}(\cEinf)
  @>{\ddb_{\alpha}}>>
  \Mat(k_0+1| k_1)\otimes\Lambda^{1}(\cEinf) \\
  @|        @VV{\nabla}V            @VV{\nabla}V \\
  \fg\otimes\Lambda^{0}(\cEinf)
  @>\boldsymbol{\varrho}>>
  \Gamma T(\Lambda^{0}(\tilde{\cE}))    
  @>{\partial_{U}}>>
  \Gamma T(\Lambda^{1}(\tilde{\cE})),
\end{CD}
\]%
where $\nabla = \boldsymbol{\varrho} \circ \rho^{-1}$ is a switch from the
representation~$\rho$ to the representation~$\boldsymbol{\varrho}$ for the Lie superalgebra~$\fg$.
\end{lemma}

\begin{proof} 
Consider any $\gamma = q^k \cdot e_k \in \fg\otimes\Lambda^{0}(\cEinf)$ and put
$\boldsymbol{\varrho}(\gamma) = X\in \Gamma T(\Lambda^{0}(\tilde{\cE}))$ and $\rho(\gamma) =
Q\in \Mat(k_0+1| k_1)\otimes\Lambda^{0}(\cEinf)$. A direct calculation shows that
\begin{align*}
(\nabla \circ {}& \boldsymbol{\partial}_{\alpha} \circ \rho) (\gamma) =  (\nabla
\circ \boldsymbol{\partial}_{\alpha} ) (Q)  = \nabla ( \bar{\Id}_hQ - \bbl \alpha,  Q
\bbr) \\
{}={}& \nabla \left( \Id x\, ( \bar{D}_x(q^k)\rho(e_k) - \bbl a^i\rho(e_i),
q^k\rho(e_k) \bbr) + \Id t\, ( \bar{D}_t(q^k)\rho(e_k) - \bbl b^j\rho(e_j),
q^k\rho(e_k) \bbr \right)   \\
{}={}& \Id x\, \left(\bar{D}_x(q^k)\boldsymbol{\varrho}(e_k) + [-a^i\boldsymbol{\varrho}(e_i), q^k\boldsymbol{\varrho}(e_k)]\right) 
+ \Id t\,(\bar{D}_t(q^k)\boldsymbol{\varrho}(e_k) + [-b^i\boldsymbol{\varrho}(e_i),
q^k\boldsymbol{\varrho}(e_k)]). \\
\intertext{On the other hand, we have that}
(\partial_U \circ {} & \boldsymbol{\varrho})(\gamma) =  \partial_U X = \fnl X, U_\lambda \fnr\\
 {} = {}&
\left[\Id x \left(\tilde{D}_x (X \inner \Id w^\ell ) - (X \inner \Id w^s)\frac{\partial w^\ell_x}{\partial w^s}\right)
+ \Id t \left(\tilde{D}_t (X \inner \Id w^\ell) - (X \inner \Id w^s)\frac{\partial w^\ell_t}{\partial
  w^s}\right)\right]\otimes\frac{\partial}{\partial w^\ell}. \\
\intertext{By using the formula $\tilde{D}_x (X\inner \Id w^\ell) = \bar{D}_x (X\inner \Id w^\ell) + w^s_x\,
\frac{\partial}{\partial w^s} (X \inner \Id w^\ell)$, we continue the equality and obtain that}
{}%
= {} & \Bigl[\Id x \left(\bar{D}_x
(X\inner \Id w^\ell) +  w^s_x\frac{\partial}{\partial w^s} (X \inner \Id w^\ell)  -
(X\inner \Id w^s)\frac{\partial w^\ell_x}{\partial w^s}  \right) \\
{} &  + \Id t \left(\bar{D}_t (X \inner \Id w^\ell)  +  w^s_t\frac{\partial}{\partial w^s}(X \inner \Id w^\ell) - (X \inner \Id w^s)\frac{\partial w^\ell_t}{\partial
  w^s}\right)\Bigr]\otimes\frac{\partial}{\partial w^\ell}  \\
= 
{} & 
  \Bigr[ 
       \Id x \left(\bar{D}_x(X)\inner \Id w^\ell  + \bbl w^s_x\frac{\partial}{\partial w^s}, X \bbr \inner \Id w^\ell \right) 
     + \Id t \left(\bar{D}_t(X)\inner \Id w^\ell +  \bbl w^s_t\frac{\partial}{\partial w^s}, X \bbr \inner \Id w^\ell \right) 
  \Bigl] 
  \otimes\frac{\partial}{\partial w^\ell}. \\
\intertext{Recall that the covering at hand is 
obtained by switching between Lie algebra's representations. 
From formulas~\eqref{covw} we infer that the equality continues,}
{}%
= {} & \Bigr[ \Id x (\bar{D}_x(q^k)\boldsymbol{\varrho}(e_k) + [-a^i\boldsymbol{\varrho}(e_i), q^k\boldsymbol{\varrho}(e_k)])\inner \Id w^\ell\\
{} &  + \Id t
(\bar{D}_t(q^k)\boldsymbol{\varrho}(e_k) + [-b^i\boldsymbol{\varrho}(e_i),
q^k\boldsymbol{\varrho}(e_k)])\inner \Id w^\ell )\Bigr]\otimes\frac{\partial}{\partial w^\ell} \\
{}  = {}& \Id x
\left(\bar{D}_x(q^k)\boldsymbol{\varrho}(e_k) + [-a^i\boldsymbol{\varrho}(e_i), q^k\boldsymbol{\varrho}(e_k)]\right) 
+ \Id t (\bar{D}_t(q^k)\boldsymbol{\varrho}(e_k) + [-b^i\boldsymbol{\varrho}(e_i),
q^k\boldsymbol{\varrho}(e_k)]).
\end{align*}
We finally obtain that $(\nabla \circ \ddb_{\alpha} \circ \rho)
(\gamma) =  (\partial_U \circ \boldsymbol{\varrho})(\gamma)$ for all~$\gamma$, which proves our claim.
\end{proof}

Now let us study in more detail the case of \emph{removable} parameters.
Let $\alpha(\lambda) = a^i \rho(e_i)\,\mathrm{d}x + b^j \rho(e_j)\,\mathrm{d}t$ be a smooth family of $\fg$-\/valued zero\/-\/curvature representations for the
system~$\cE$ but let the parameter $\lambda\in\mathcal{I}\subseteq\BBC$ be removable by using a smooth family of gauge transformations~$S_\lambda$.
By Remark~\ref{RemAlmostConverse}, %
there are $\fg$-\/matrices~$Q = q^k\rho(e_k)$ such that
\[%
\frac{\mathrm{d}}{\mathrm{d}\lambda} \alpha = \bar{\Id}_h Q - [\alpha, Q].
\]%
In components, we have that%
\begin{subequations}\label{MarvanEqComponent}
\begin{align}
    \frac{\mathrm{d}}{\mathrm{d}\lambda}(a^i) \rho(e_i) &=  \bar{D}_x (q^k)\rho(e_k) - a^iq^k\rho([e_i, e_k]),\\
    \frac{\mathrm{d}}{\mathrm{d}\lambda}(b^j) \rho(e_i) &=  \bar{D}_t (q^k)\rho(e_k) - b^jq^k\rho([e_j, e_k]).
\end{align}
\end{subequations}
By virtue of the representation~$\boldsymbol{\varrho}$, at every~$\lambda$
the $\fg$-matrix $Q=q^k \rho(e_k)$ determines the vector field $X =
q^k \boldsymbol{\varrho}(e_k)$ on~$\tilde{\cE}$.

The following proposition is a regular generator of solutions for
equation~\eqref{eqFNXshadow} in the case of coverings derived from 
zero\/-\/curvature representations with removable parameters.
It was remarked 
in~\cite{KKIgonin2003} that the formalism of zero\/-\/curvature representations and 
their parametric families can be viewed as a special case of 
the Fr\"olicher\/--\/Nijenhuis bracket formalism for deformations of coverings 
of generic %
nature; we thus substantiate that claim 
from \textit{loc.\ cit.} by giving an explicit proof.

\begin{prop}\label{propFNMarvan}
The vector field~$X=q^k \boldsymbol{\varrho}(e_k)$ satisfies structure equation~\eqref{eqFNXshadow}.
\end{prop}

\begin{proof} 
From Lemma~\ref{lemmaCommute} we infer that
\begin{align*}
\fnl X, U_\lambda \fnr  = {} & \Bigr[ \Id x \left(\bar{D}_x(q^k)\boldsymbol{\varrho}(e_k) + [-a^i\boldsymbol{\varrho}(e_i), q^k\boldsymbol{\varrho}(e_k)]\right)\inner \Id w^\ell\\
{} &  + \Id t
(\bar{D}_t \left(q^k)\boldsymbol{\varrho}(e_k) + [-b^i\boldsymbol{\varrho}(e_i),
q^k\boldsymbol{\varrho}(e_k)]\right)\inner \Id w^\ell \Bigr]\otimes\frac{\partial}{\partial w^\ell}.\\
\intertext{Using~\eqref{MarvanEqComponent}, we obtain that}
{}& \Bigl[ \Id x\,\frac{\Id}{\Id \lambda}(a^i)(\boldsymbol{\varrho}(e_i)\inner
  \Id w^\ell) + \Id t\,\frac{\Id}{\Id \lambda}(b^i)(\boldsymbol{\varrho}(e_i)\inner \Id w^\ell)
  \Bigr]\otimes\frac{\partial}{\partial w^\ell} \\
  {} = {}& \left[ -\frac{\Id}{\Id \lambda} (w^\ell_x)\,\Id x -\frac{\Id}{\Id \lambda} (w^\ell_t)\,\Id t \right] \otimes
  \frac{\partial}{\partial w^\ell} = \frac{\Id}{\Id \lambda} U_\lambda.
\end{align*}
This proves that the vector field $X$ is 
a solution of equation~\eqref{eqFNXshadow}.
\end{proof}

\begin{rem}
This proof can easily be extended to the case of any finite~$n$.
\end{rem}

We refer to~\cite{Catalano2015} for discussion of the set\/-\/up when a symmetry $X\in\Gamma (T\cEinf)$ of an equation~$\cE$ can be used to extend a given zero\/-\/curvature representation for~$\cE$ to a ``nontrivial'' family of zero\/-\/curvature
representations. We expect that this result must have a straightforward generalisation to the $\BBZ_2$-\/graded case.

\begin{cor}\label{corFNrem}
Let $\tau\colon\tilde{\cE}_\lambda\to\cE$ be a family of coverings %
for the family of $\fg$-\/valued zero\/-\/curvature
representations~$\alpha_\lambda$ \textup{(}smoothly depending on the
parameter~$\lambda$\textup{)}.
If equation~\eqref{eqFNXshadow} has no $\tau$-\/vertical solutions,
then the parameter~$\lambda$ in~$\alpha_\lambda$ is not removable by gauge transformations.
\end{cor}

\begin{example}%
Let us illustrate the claim of Proposition~\ref{propFNMarvan}. Namely, let
us construct a $(1|1)$-\/dimensional covering over the $N{=}2$,
$a{=}4$ SKdV equation~\eqref{SKdVComponents} by taking %
the $\gsl(2|1)$\/-\/valued 
zero\/-\/curvature representation~$\beta$ from %
Example~\ref{exN2ZCRRem} on p.~\pageref{exN2ZCRRem}. 
Using representation~$\boldsymbol{\varrho}$ from Example~\ref{exProjSub}, we obtain
\begin{align*}
w_x = {} & \lambda^2 + 2\lambda w  + w^2 + u_{0}^2 + u_{12}  - f_{2}u_{2}
+  \ii f_{2}u_{1},\\
f_x = {} & \lambda f_{2} + f_{2}w  +  \ii  f_{2}u_{0}  + u_{2} +  \ii  u_{1}, \\
w_t = {} & 2\lambda^2(2u_{0}^2 - u_{12} ) + \lambda(8w u_{0}^2 - 4w
u_{12}  + 8u_{0} u_{0;x} - 2u_{12;x} + fu_{2;x} - \ii fu_{1;x}  \\
{}&{}+ 4\ii fu_{2}u_{0}   + 4fu_{1}u_{0} ) + 4w^2u_{0}^2 - 2w^2u_{12}  +
8w u_{0} u_{0;x} - 2w u_{12;x} + 4u_{0}^4 + 2u_{0}^2u_{12} \\
{}&{} + 4u_{0} u_{0;xx} - 2u_{12}^2 + 4u_{0;x}^2 - u_{12;xx} + fw
u_{2;x} - \ii fw u_{1;x}  + 5\ii fu_{0} u_{2;x}  + 5fu_{0} u_{1;x}\\
{}&{} + fu_{2;xx} - \ii fu_{1;xx}  + 4\ii fu_{2}w u_{0}   - 8fu_{2}u_{0}^2 +
2fu_{2}u_{12}  + 4\ii fu_{2}u_{0;x}  + 4fu_{1}w u_{0} \\
{}&{}  + 8\ii fu_{1}u_{0}^2  -
2\ii fu_{1}u_{12}   + 4fu_{1}u_{0;x} + u_{2}u_{2;x} + 8u_{2}u_{1}u_{0}  +
u_{1}u_{1;x},\\
f_t = {} & \lambda( - u_{2;x} - \ii u_{1;x}  + 4fu_{0}^2 - 2fu_{12}  +
4\ii u_{2}u_{0}   - 4u_{1}u_{0} ) - w u_{2;x} - \ii w u_{1;x}  
+ 5\ii u_{0} u_{2;x}  \\
{} & {} - 5u_{0} u_{1;x} - u_{2;xx} - \ii u_{1;xx}  + 4fw u_{0}^2 - 2fw
u_{12}  + 4\ii fu_{0}^3  - 6\ii fu_{0} u_{12}   + 4fu_{0} u_{0;x} \\ 
{}&{}- fu_{0;xx} \ii  - fu_{12;x} - 4fu_{2}u_{1} \ii  + 4u_{2}w u_{0}  \ii  +
8u_{2}u_{0}^2 - 2u_{2}u_{12}  + 4\ii u_{2}u_{0;x}  - 4u_{1}w u_{0} \\
{}&{} + 8\ii u_{1}u_{0}^2  - 2\ii u_{1}u_{12}   - 4u_{1}u_{0;x}.
\end{align*}
In agreement with Proposition~\ref{propFNMarvan}, we find the solution $X = {\partial}/{\partial w}$ of
equation~\eqref{eqFNXshadow}:\ in\-deed,\ this field is obtained from the $\gsl(2|1)$\/-matrix~$Q$ which we
introduced in Example~\ref{exN2ZCRRem}.
Let us finally note that this example of vector field~$X$ is \emph{not} the infinitesimal generator of a scaling.
\end{example}

\begin{example}\label{Ex98}
  An inverse recursion operator has been obtained for fifth-\/order evolution superequation~\eqref{eq5ordF} in
  components %
  in~\cite{cookbook}.\footnote{A recursion operator, formulated for~\eqref{eq5ordF} in terms of superfields
    and superderivatives, was conjectured in~\cite{TianLiu5ord}.}
A known relation between the inverse recursion operators and zero-\/curvature representations
(see~\cite{BaranMarvan2006} for details) yields a new family of $\fgl(9|8)$-\/valued zero-\/cur\-va\-tu\-re
representations for~\eqref{eq5ordF}.
Denote by~$\lambda$ the parameter under study.
We now realise this family of zero\/-\/curvature representations as a family of 
$(9|8)$-\/dimensional linear coverings over~\eqref{eq5ordF}. 
Seventeen new nonlocalities are introduced;
the variables $S$, $W_1$, $W_2$, $W_3$, $V_1$, $\ldots$, $V_5$ are parity-\/even 
and the variables $Q_1$, $Q_2$, $Q_3$, $O_1$, $\ldots$, $O_5$ are parity-\/odd.
At every $\lambda$, the derivatives of the new %
variables are given by the formulas
\begin{align*}
  W_{3;x} = {}&{}  - O_3 \xi  - 3O_1 \xi u  - 9Q_2 \xi u  + 3Q u \xi_{x} + 2V_2 u  + S (6u^2 + 3\xi \xi_{x}),\\
  W_{2;x} = {}&{} 2O_1 \xi  + 2S u,\\
  W_{1;x} = {}& S,\quad
  S_x = {} V_1,\quad
  V_{1;x} ={} V_2,\quad
  V_{2;x} ={} V_3,\quad
  V_{3;x} ={} V_4,\quad
  V_{4;x} ={} V_5,\\
  V_{5;x} ={}& 6O_4 \xi_{x} + O_3 (3\xi_{xx} - 18\xi u ) + O_2 (6u \xi_{x} - 4\xi_{xxx} - 29\xi u_{x}) + O_1
               (6u \xi_{xx} - 4\xi_{4x}\\
  {}&{}+ 11\xi_{x}u_{x} - 36\xi u^2 - 15\xi u_{xx}) + ( - 3Q_3 \xi_{x})/2 + Q_2 (21u \xi_{xx} + 3\xi_{4x} + 21\xi_{x}u_{x} + 9\xi u^2)\\
  {}&{}+ Q (51u^2\xi_{x} + 12u \xi_{xxx} - 2\xi_{5x} + 26\xi_{xx}u_{x} + 22\xi_{x}u_{xx} - 18\xi u u_{x} + 2\xi u_{xxx})/2 - 12V_4 u\\
  {}&{}- 30V_3 u_{x} + V_2 ( - 48u^2 - 40u_{xx} + 15\xi \xi_{x}) + V_1 ( - 140u u_{x} - 30u_{xxx} + 27\xi \xi_{xx}) - 2W_3 u_{x}\\
  {}&{}+ W_2 ( - 12u u_{x} - 2u_{xxx} - 3\xi \xi_{xx}) + W_1 ( - 60u^2u_{x} - 20u u_{xxx} - 2u_{5x} - 40u_{xx}u_{x}\\
  {}&{}+ 10\xi_{xxx}\xi_{x} + 30\xi u \xi_{xx} + 30\xi \xi_{x}u_{x}) + S ( - 64u^3 - 96u u_{xx} - 12u_{4x} - 72u_{x}^2 + 8\xi_{xx}\xi_{x}\\
  {}&{}+ 90\xi u \xi_{x} + 15\xi \xi_{xxx} - \lambda ),\\
  Q_{3;x} ={}& 2O_3 u  + 7O_1 u^2 + Q_2 (6u^2 - 6\xi \xi_{x}) - 2V_3 \xi  - 6W_2 \xi u  + 14S u \xi_{x},\\
  Q_{2;x} ={}& Q u  + S \xi,\\
  Q_x = {}& O_1,\quad
  O_{1;x} ={} O_2,\quad
  O_{2;x} ={} O_3,\quad
  O_{3;x} ={} O_4,\quad
  O_{4;x} ={} O_5,\\
  O_{5;x} ={}& - 12O_4 u  - 24O_3 u_{x} + O_2 ( - 27u^2 - 19u_{xx} - 2\xi \xi_{x}) + O_1 ( - 63u u_{x} - 7u_{xxx})\\
  {}&{}- 3/2Q_3 u + Q_2 ( - 21u u_{x} - 3u_{xxx}) + Q ( - 35u^3 - 48u u_{xx} - 2u_{4x} - 40u_{x}^2 + 6\xi_{xx}\xi_{x}\\
  {}&{}- 4\xi \xi_{xxx} - 2\lambda )/2 - 6V_3 \xi_{x} + V_2 ( - 21\xi_{xx} - 21\xi u ) + V_1 ( - 50u \xi_{x} - 23\xi_{xxx} - 27\xi u_{x})\\
  {}&{}- 2W_3 \xi_{x} + W_2 (6u \xi_{x} + \xi_{xxx} + 3\xi u_{x}) + W_1 ( - 30u^2\xi_{x} - 20u \xi_{xxx} - 2\xi_{5x} - 30\xi_{xx}u_{x}\\
  {}&{}- 10\xi_{x}u_{xx} - 30\xi u u_{x}) + S ( - 57u \xi_{xx} - 11\xi_{4x} - 52\xi_{x}u_{x} - 36\xi u^2 - 15\xi u_{xx}),\\
\end{align*}
\begin{align*}
  S_t = {}& - 5O_3 \xi_{x} + 15O_2 \xi u  + O_1 (5\xi_{xxx} + 15\xi u_{x}) + Q ( - 15u \xi_{xx} - 15\xi_{x}u_{x}) + V_5  + 10V_3 u\\
  {}&{}+ 20V_2 u_{x} + V_1 (30u^2 + 20u_{xx} - 15\xi \xi_{x}) + S (60u u_{x} + 10u_{xxx} - 15\xi \xi_{xx}),\\
  Q_t = {}&  O_5  + 10O_3 u  + 15O_2 u_{x} + O_1 (15u^2 + 5u_{xx}) + 15Q u u_{x} + 5V_2 \xi_{x} + V_1 (15\xi_{xx} + 15\xi u )\\
  {}&{}+ S (30u \xi_{x} + 10\xi_{xxx} + 15\xi u_{x}),\\
  V_{1;t} ={}& (S_t)_x,\quad
  V_{2;t} ={} (V_{1;t})_x,\quad
  V_{3;t} ={} (V_{2;t})_x,\quad
  V_{4;t} ={} (V_{3;t})_x,\quad
  V_{5;t} ={} (V_{4;t})_x,\\
  O_{1;t} ={}& (Q_t)_x,\quad
  O_{2;t} ={} (O_{1;t})_x,\quad
  O_{3;t} ={} (O_{2;t})_x,\quad
  O_{4;t} ={} (O_{3;t})_x,\quad
  O_{5;t} ={} (O_{4;t})_x,\\
  W_{1;t} = {}& - 5O_2 \xi_{x} + O_1 (5\xi_{xx} + 15\xi u ) - 15Q u \xi_{x} + V_4  + 10V_2 u  + 10V_1 u_{x}\\
  {}&{}+ S (30u^2 + 10u_{xx} - 15\xi \xi_{x}),\\
  W_{2;t} = {}& 2O_5 \xi  - 2O_4 \xi_{x} + O_3 (2\xi_{xx} + 20\xi u ) + O_2 ( - 30u \xi_{x} - 2\xi_{xxx} + 30\xi u_{x}) + O_1 (30u \xi_{xx}\\
  {}&{}+ 2\xi_{4x} + 60\xi u^2 + 10\xi u_{xx}) + Q ( - 30u^2\xi_{x} + 30\xi u u_{x}) + 2V_4 u  - 2V_3 u_{x} + V_2 (20u^2\\
  {}&{}+ 2u_{xx} - 10\xi \xi_{x}) + V_1 ( - 2u_{xxx} - 30\xi \xi_{xx}) + S (60u^3 + 40u u_{xx} + 2u_{4x} - 30\xi_{xx}\xi_{x}\\
  {}&{}- 120\xi u \xi_{x} - 20\xi \xi_{xxx}),\\
  Q_{2;t} = {}&  O_4 u  - O_3 u_{x} + O_2 (10u^2 + u_{xx} + 5\xi \xi_{x}) + O_1 ( - 5u u_{x} - u_{xxx} - 5\xi \xi_{xx}) + Q (15u^3 \\
  {}&{}+ 10u u_{xx} + u_{4x} + 5u_{x}^2 - 5\xi_{xx}\xi_{x}) + V_4 \xi  - V_3 \xi_{x} + V_2 (\xi_{xx} + 10\xi u ) + V_1 ( - 5u \xi_{x}\\
  {}&{}- \xi_{xxx} + 10\xi u_{x}) + S (20u \xi_{xx} + \xi_{4x} - 5\xi_{x}u_{x} + 45\xi u^2 + 10\xi u_{xx}), \\
  W_{3;t} ={}& O_5 ( - \xi_{xx} - \xi u ) + O_4 (6u \xi_{x} + \xi_{xxx} + 4\xi u_{x}) + O_3 ( - 20u \xi_{xx} - \xi_{4x} + 2\xi_{x}u_{x} - 30\xi u^2\\
  {}&{}- 5\xi u_{xx}) + O_2 (54u^2\xi_{x} + 8u \xi_{xxx} - 15\xi_{xx}u_{x} - 2\xi_{x}u_{xx} + 11\xi u u_{x} + 4\xi u_{xxx})\\
  {}&{}+ O_1 ( - 57u^2\xi_{xx} - 4u \xi_{4x} + 4u \xi_{x}u_{x} - \xi_{xxx}u_{x} - 2\xi_{xx}u_{xx} - 2\xi_{x}u_{xxx} - 89\xi u^3\\
  {}&{}- 9\xi u u_{xx} - 4\xi u_{x}^2 - 3\xi \xi_{xx}\xi_{x} + \xi \lambda ) + Q_3 ( - 9u \xi_{x} + 3\xi u_{x})/2 + Q_2 ( - 48u^2\xi_{xx}\\
  {}&{}- 3u \xi_{4x} + 66u \xi_{x}u_{x} + 9\xi_{xxx}u_{x} - 9\xi_{xx}u_{xx} + 6\xi_{x}u_{xxx} - 108\xi u^3 - 69\xi u u_{xx} - 6\xi u_{4x}\\
  {}&{}- 24\xi u_{x}^2 + 45\xi \xi_{xx}\xi_{x}) + Q (157u^3\xi_{x} + 42u^2\xi_{xxx} + 2u \xi_{5x} - 32u \xi_{xx}u_{x} + 38u \xi_{x}u_{xx}\\
  {}&{}- 6\xi_{4x}u_{x} + 6\xi_{xxx}u_{xx} - 6\xi_{xx}u_{xxx} + 4\xi_{x}u_{4x} + 44\xi_{x}u_{x}^2 - 2\xi_{x}\lambda  + 111\xi u^2u_{x} + 28\xi u u_{xxx}\\
  {}&{}+ 2\xi u_{5x} + 38\xi u_{xx}u_{x} - 14\xi \xi_{xxx}\xi_{x})/2 - 2V_5 u_{x} + V_4 (2u^2 + 2u_{xx} + 2\xi \xi_{x})\\
  {}&{}+ V_3 ( - 32u u_{x} - 2u_{xxx} - 5\xi \xi_{xx}) + V_2 (24u^3 + 32u u_{xx} + 2u_{4x} - 28u_{x}^2 + \xi_{xx}\xi_{x}\\
  {}&{}+ 15\xi u \xi_{x} + 4\xi \xi_{xxx}) + V_1 ( - 100u^2u_{x} - 12u u_{xxx} - 16u_{xx}u_{x} - 7\xi_{xxx}\xi_{x} - 50\xi u \xi_{xx}\\
  {}&{}- 2\xi \xi_{4x} + 18\xi \xi_{x}u_{x}) + W_3 ( - 4u u_{x} - 2\xi \xi_{xx}) + W_2 ( - 24u^2u_{x} - 4u u_{xxx} + \xi_{xxx}\xi_{x}\\
  {}&{}+ \xi \xi_{4x} + 12\xi \xi_{x}u_{x}) + W_1 ( - 120u^3u_{x} - 40u^2u_{xxx} - 4u u_{5x} - 80u u_{xx}u_{x} - 2\xi_{5x}\xi_{x}\\
  {}&{}- 30\xi_{xx}\xi_{x}u_{x} + 30\xi u^2\xi_{xx} - 20\xi u \xi_{4x} - 60\xi u \xi_{x}u_{x} - 2\xi \xi_{6x} - 50\xi \xi_{xxx}u_{x} - 40\xi \xi_{xx}u_{xx}\\
  {}&{}- 10\xi \xi_{x}u_{xxx}) + S (52u^4 + 108u^2u_{xx} + 8u u_{4x} - 144u u_{x}^2 + 97u \xi_{xx}\xi_{x} - 2u \lambda\\
  {}&{}- 32u_{xxx}u_{x} + 16u_{xx}^2 + 5\xi_{4x}\xi_{x} - 16\xi_{xxx}\xi_{xx} + 192\xi u^2\xi_{x} + 33\xi u \xi_{xxx} - 46\xi \xi_{xx}u_{x}\\
  {}&{}+ 19\xi \xi_{x}u_{xx}),
\end{align*}      
\begin{align*}
Q_{3;t} ={}&  O_5 (3u^2 + 2u_{xx} + 2\xi \xi_{x}) + O_4 ( - 12u u_{x} - 2u_{xxx} - 2\xi \xi_{xx}) + O_3 (46u^3 + 38u u_{xx}\\
  {}&{}+ 2u_{4x} + 12u_{x}^2 - 14\xi_{xx}\xi_{x} + 6\xi u \xi_{x} - 2\xi \xi_{xxx}) + O_2 ( - 59u^2u_{x} - 16u u_{xxx} + 26u_{xx}u_{x}\\
  {}&{}+ 2\xi_{xxx}\xi_{x} - 10\xi u \xi_{xx} + 4\xi \xi_{4x} + 66\xi \xi_{x}u_{x}) + O_1 (109u^4 + 121u^2u_{xx} + 8u u_{4x} - 12u u_{x}^2\\
  {}&{}- 12u \xi_{xx}\xi_{x} - 2u \lambda  - 10u_{xxx}u_{x} + 12u_{xx}^2 + 2\xi_{4x}\xi_{x} - 10\xi_{xxx}\xi_{xx} - 78\xi u^2\xi_{x}\\
  {}&{}- 18\xi u \xi_{xxx} - 54\xi \xi_{xx}u_{x} + 14\xi \xi_{x}u_{xx}) - 3Q_3 \xi \xi_{xx} + Q_2 (72u^4 + 78u^2u_{xx} + 6u u_{4x}\\
  {}&{}- 138u \xi_{xx}\xi_{x} - 6u_{xxx}u_{x} + 6u_{xx}^2 - 6\xi_{4x}\xi_{x} + 12\xi_{xxx}\xi_{xx} - 198\xi u^2\xi_{x} - 18\xi u \xi_{xxx}\\
  {}&{}- 6\xi \xi_{xx}u_{x} + 12\xi \xi_{x}u_{xx}) + Q ( - 67u^3u_{x} - 30u^2u_{xxx} - 2u u_{5x} + 4u u_{xx}u_{x} + 10u \xi_{xxx}\xi_{x}\\
  {}&{}+ 2u_{4x}u_{x} + 12u_{x}^3 + 2u_{x}\lambda  - 2\xi_{5x}\xi_{x} - 64\xi_{xx}\xi_{x}u_{x} + 153\xi u^2\xi_{xx} - 10\xi u \xi_{4x} + 66\xi u \xi_{x}u_{x}\\
  {}&{}- 2\xi \xi_{6x} - 54\xi \xi_{xxx}u_{x} - 36\xi \xi_{xx}u_{xx} - 10\xi \xi_{x}u_{xxx}) + V_5 ( - 2\xi_{xx} + 4\xi u ) + V_4 (8u \xi_{x}\\
  {}&{}+ 2\xi_{xxx} + 4\xi u_{x}) + V_3 ( - 38u \xi_{xx} - 2\xi_{4x} - 12\xi_{x}u_{x} + 24\xi u^2) + V_2 (79u^2\xi_{x} + 16u \xi_{xxx}\\
  {}&{}- 40\xi_{xx}u_{x} + 14\xi_{x}u_{xx} + 120\xi u u_{x}) + V_1 ( - 123u^2\xi_{xx} - 12u \xi_{4x} - 52u \xi_{x}u_{x} + 24\xi_{xxx}u_{x}\\
  {}&{}- 12\xi_{xx}u_{xx} - 14\xi_{x}u_{xxx} + 101\xi u^3 + 142\xi u u_{xx} + 4\xi u_{4x} + 64\xi u_{x}^2 - 82\xi \xi_{xx}\xi_{x} + 2\xi \lambda )\\
  {}&{}+ W_3 ( - 4u \xi_{xx} + 4\xi u_{xx}) + W_2 ( - 48u^2\xi_{xx} - 4u \xi_{4x} + 12u \xi_{x}u_{x} + 4\xi_{xxx}u_{x} - 6\xi_{xx}u_{xx}\\
  {}&{}+ 2\xi_{x}u_{xxx} - 72\xi u^3 - 30\xi u u_{xx} - 2\xi u_{4x} - 12\xi u_{x}^2 + 18\xi \xi_{xx}\xi_{x}) + W_1 ( - 60u^3\xi_{xx}\\
  {}&{}- 40u^2\xi_{4x} - 240u^2\xi_{x}u_{x} - 4u \xi_{6x} - 60u \xi_{xxx}u_{x} - 80u \xi_{xx}u_{xx} - 60u \xi_{x}u_{xxx} + 4\xi_{5x}u_{x}\\
  {}&{}+ 60\xi_{xx}u_{x}^2 - 4\xi_{x}u_{5x} - 60\xi_{x}u_{xx}u_{x} + 60\xi u^2u_{xx} + 40\xi u u_{4x} + 240\xi u u_{x}^2 + 120\xi u \xi_{xx}\xi_{x}\\
  {}&{}+ 4\xi u_{6x} + 120\xi u_{xxx}u_{x} + 80\xi u_{xx}^2 - 20\xi \xi_{4x}\xi_{x} - 20\xi \xi_{xxx}\xi_{xx}) + S (358u^3\xi_{x}\\
  {}&{}+ 138u^2\xi_{xxx} + 8u \xi_{5x} - 192u \xi_{xx}u_{x} + 244u \xi_{x}u_{xx} - 12\xi_{4x}u_{x} + 34\xi_{xxx}u_{xx} - 34\xi_{xx}u_{xxx}\\
  {}&{}+ 10\xi_{x}u_{4x} - 12\xi_{x}u_{x}^2 - 2\xi_{x}\lambda  + 549\xi u^2u_{x} + 102\xi u u_{xxx} + 8\xi u_{5x}\\
  {}&{}+ 224\xi u_{xx}u_{x} - 48\xi \xi_{xxx}\xi_{x}),
\end{align*}
For this family of coverings over~\eqref{eq5ordF}, 
the solution of equation~\eqref{eqFNXshadow} is given by the vector field
\begin{multline}\label{X5ord}
  X = \lambda^{-1}\Bigl(
  x\frac{\dd}{\dd x} + 5t\frac{\dd}{\dd t}
  - 2u\frac{\dd}{\dd u} - \tfrac32\xi \frac{\dd}{\dd \xi}
  +\ldots %
  - \tfrac92 S \frac{\dd}{\dd S} + 5 Q \frac{\dd}{\dd Q}
  + \tfrac{11}{2}W_1\frac{\dd}{\dd W_1} + \tfrac72 W_2\frac{\dd}{\dd W_2}
\\
{}+ \tfrac32 W_3\frac{\dd}{\dd W_3}
  + 4 Q_2\frac{\dd}{\dd Q_2} + Q_3\frac{\dd}{\dd Q_3}
  + \tfrac72 V_1\frac{\dd}{\dd V_1} +\tfrac52 V_2\frac{\dd}{\dd V_2}
  + \tfrac32 V_3\frac{\dd}{\dd V_3} +\tfrac12 V_4\frac{\dd}{\dd V_4}
  - \tfrac12 V_5\frac{\dd}{\dd V_5}\\
{}  + 4O_1\frac{\dd}{\dd O_1} + 3O_2\frac{\dd}{\dd O_2}
  + 2O_3\frac{\dd}{\dd O_3} + O_4\frac{\dd}{\dd O_4}
  \Bigr).
\end{multline}
It is clear that vector field~\eqref{X5ord} is not $\tau$-\/verticalisable.
\end{example}

\begin{claim}\label{pClaim}%
In the covering~$\tau$ constructed in Example~\ref{Ex98}
for super\/-\/equation~\eqref{eq5ordF},
deformation equation~\eqref{eqFNXshadow} does not admit \emph{any} $\tau$-\/vertical solutions.
\end{claim}

\begin{proof}[Sketch of the proof] 
Equation~\eqref{eqFNXshadow}, viewed for this covering as a differential equation with respect to components of the vector field~$X$, is a system of %
linear inhomogeneous equations in total %
derivatives. 
The corresponding \emph{homogeneous} system describes symmetries of this covering. 
Therefore, the general solution~$X_{\text{gen}}$ of~\eqref{eqFNXshadow} has the form
$X_{\text{gen}} = X + X_{\text{sym}}$, where $X$~is given by~\eqref{X5ord} 
and a solution~$X_{\text{sym}}$ of the %
homogeneous system is a symmetry of the covering. 
The $\tau$-\/horizontal part of~\eqref{X5ord} is the scaling symmetry of~\eqref{eq5ordF}.
By a straightforward calculation we have established 
that the scaling symmetry of~\eqref{eq5ordF} cannot be lifted
to a symmetry of the covering~$\tau$. 
Consequently, all solutions of~\eqref{eqFNXshadow} %
have a nonzero $\tau$-\/horizontal part, hence none of them is $\tau$-\/verticalisable.
\end{proof}

The above claim and Corollary~\ref{corFNrem} combined yield that
the parameter~$\lambda$ in the covering in Example~\ref{Ex98} 
is not removable by gauge transformations,
and therefore, equation~\eqref{eq5ordF} is %
integrable.

\section*{Conclusion}\label{pConclusion}\noindent%
We extended --~to the $\BBZ_2$-graded case~-- Marvan's method for inspecting the 
(non)\/re\-mo\-va\-bi\-li\-ty of a parameter under the action of a smooth family of gauge transformations on a given family of zero\/-\/curvature representations. %
This generalisation of the standard technique can be used further for solution of %
Gardner's deformation problems for the $N{=}2$-\/supersymmetric %
KdV equations and other $\mathbb{Z}_2$-\/graded completely integrable systems.
At the same time, we confirmed that a switch %
between the representations of Lie (super-)\/algebras establishes 
a link between the two classes of nonlocal geometries and also between the arising differentials. %
In particular, by analysing %
this relation in the case of zero\/-\/curvature representations with removable parameters~$\lambda$, we explicitly described the equivalence classes of $\tau_\lambda$-\/shadows that determine, by virtue of 
structure equation~\eqref{eqFNXshadow}, the evolution of Cartan's structural elements in families of coverings~$\tau_\lambda$.

\smallskip
We remember that the technique for calculation of the horizontal gauge cohomology groups (see Lemma~\ref{lemZ2CoHom})
constitutes %
another result of the original papers~\cite{Marvan2002,Marvan2010}.
Namely, suppose that $\alpha_\lambda$~is a family, depending on a
parameter~$\lambda$, of $\fg$\/-\/valued zero-curvature representation for an equation~$\cE$. In~\cite{Marvan2002}, the horizontal gauge cohomology
complex~$\bar{H}^q_{\alpha_\lambda}(\cE,\fg)$ was associated\footnote{%
This cohomology theory %
is helpful %
for solution of a different %
problem, namely, \emph{construction} of %
parametric families of zero\/-\/curvature representations with nonremovable parameters (see~\cite{Marvan2010} for details).}
with such family~$\alpha_\lambda$. 
It is standard that the first horizontal gauge cohomology group~$\bar{H}^1_{\alpha_\lambda}(\cE,\fg)$ contains the obstructions to removability of the parameter~$\lambda$ (cf.\ section~\ref{secZCRCovering} above 
and~\cite{KKIgonin2003,JKIgonin2002}). 
However, calculating %
the cohomology group~$\bar{H}^1_{\alpha_\lambda}(\cEinf,\fg)$ is, in general, harder than solving equation~\eqref{EqToSolve} from Proposition~\ref{PropNonRemGrad} whenever the (non)\/removability of a given parameter~$\lambda$ is examined.

Marvan's technique for calculation of the first horizontal gauge cohomology group $\bar{H}^1_{\alpha_\lambda}(\cE,\fg)$ is based on finding 
$\fg\ltimes\fg$-\/valued zero\/-\/curvature representations for the system~$\cE$. 
Let us keep in mind that an efficient approach to finding zero\/-\/curvature representations~$\alpha$ for purely bosonic PDEs~$\cE$ was known from~\cite{Marvan1992,MarvanSL2}: it involves the use of such auxiliary structures as the characteristic elements~$\chi(\alpha)$ and then, consideration of the Jordan normal forms for the $\fg$-\/matrices contained therein.
Consequently, a proper $\BBZ_2$-\/graded generalisation of the concept of Jordan normal forms was indispensable, to make that method work in the larger set\/-\/up. Such generalisation has become available from the extended edition~\cite[\S\,D7.2]{Berezin2ed} of~\cite{Berezin}, see also~\cite{ShanderArXiv}.

\appendix
\section{Proof of Proposition~\protect{\ref{PropNonRemDas}}}\label{AppClaimZCRDas}\noindent%
To verify the claim, let us first introduce some helpful notation.
The number $k$ is called the differential order of a  differential  function $f(x,t,[u^j,\xi^l])$ with respect to the 
variable~$u^j$ if following conditions hold:
\begin{enumerate}
\item the function~$f$ essentially depends on the $k$th order derivative of~$u^j$ with respect to~$x$:
  \[ \frac{\dd f}{\dd u^j_\sigma} \neq 0, \qquad \sigma = (x\dots x),\   |\sigma| = k; \]
\item the function~$f$ does not depend on derivatives of~$u^j$ %
of order higher than $k$ with respect to~$x$:
\[ \forall p >k,\quad  \frac{\dd f}{\dd u^j_\sigma} = 0, \qquad \sigma = (x\dots x),\ |\sigma|= p.  \]
\end{enumerate}
We denote by $\dord_x^{u^j} (f)$  the differential order of a given function $f$ with respect to $u^j$. In the
same manner we define the differential order $\dord_{x}^{\xi^l}(f)$ of differential function $f(x,t, [u^j,\xi^l])$ with respect to the
odd variables~$\xi^l$.

For %
$N{=}2$, $a{=}4$-SKdV equation~\eqref{SKdVComponents} we have that ${u^1=u_0}$, ${u^2 = u_{12}}$, ${\xi^1=u_1}$, ${\xi^2=u_2}$. The maximum of
four numbers $\dord_x^{u_0} (f)$, $\dord_x^{u_1} (f)$, $\dord_x^{u_2} (f)$, and $\dord_x^{u_{12}} (f)$ is
called the \emph{differential order} of the function $f(x,t,[u_0,u_1,u_2,u_{12}])$, %
denoted by $\dord_x (f)$.
The maximum of $\dord_x (a_{ij})$, where $a_{ij}$'s are the entries of a given matrix $A$ with differential\/-\/functional coefficients, is %
the differential order of the matrix $A$; it is denoted by~$\dord_x (A)$.
Obviously, the following formulas hold:
\begin{align*}
0 \leqslant \dord_x (f + g) \leqslant  {}& \max\{ \dord (f), \dord (g)\},\\
0 \leqslant \dord_x (f\cdot g) \leqslant {}& \max\{ \dord (f), \dord (g)\},\\
\dord_x (\bar{D}_x f)  = {}& \begin{cases} 
  \dord_x (f) + 1 &\text{ if } f = f(x,t,[u_0,u_1,u_2,u_{12}]),\\
  0 &\text{ if } f = f(x,t).
\end{cases}
\end{align*}

Now, for the matrix~\eqref{eqDasA} we have that
\[
\dord_x A = 0, \qquad \dord_x \tfrac{\dd}{\dd \veps} A = 0.
\]
Let us calculate the maximal differential order of left-hand side of equation~\eqref{eqDasQA},
\[
 \dord_x\left(\tfrac{\dd}{\dd \veps} A  + \bbl A, Q \bbr \right)  \leqslant \dord_x (Q).
\]
The differential orders of the right-hand side and the left-hand side of equation~\eqref{eqDasQA}
must coincide. Therefore, we have that
\[
\dord_x (\bar{D}_x Q) = \dord_x (Q).
\]
This equality holds only in the case when all entries of the matrix $Q$ do not depend on $u_0$, $u_1$, $u_2$,
and $u_{12}$. This implies that the total derivative $\bar{D}_x$ in equation~\eqref{eqDasQA} amounts to the
partial derivative $%
{\dd}/{\dd x}$. We finally obtain the following system of equations for the entries~$q_{ij}$ of
the matrix~$Q$:
\begin{subequations}\label{eqQbigsys}
  \begin{align}
    \tfrac{\dd}{\dd x} q_{11} = {}& - u_{2}q_{31}\veps^{-1}  - \ii u_{1}q_{31}\veps^{-1}  + u_{0}^2q_{21}\veps^{-1}  
      - \ii u_{0} q_{21}\veps^{-2}  + u_{12} q_{21}\veps^{-1}  +  q_{12}\veps,\\
   \tfrac{\dd}{\dd x} q_{12} = {}&  - u_{2}(q_{32} + q_{13}\veps )\veps^{-1} + \ii u_{1}( - q_{32} + q_{13}\veps )\veps^{-1} 
      + u_{0}^2( - \veps q_{11} + \veps q_{22} - 1)\veps^{-2} \notag\\
      {} & {} + \ii u_{0} (\veps q_{11} - \veps q_{22} + 2)\veps^{-3} + u_{12} ( - \veps q_{11} + \veps q_{22} - 1)\veps^{-2} 
      + q_{12}\veps^{-1},\\
   \tfrac{\dd}{\dd x} q_{13} = {}& u_{2}( - \veps q_{22} + 1)\veps^{-2} + \ii u_{1}( - \veps q_{22} + 1 )\veps^{-2} 
      +  u_{0}^2q_{23}\veps^{-1}  + \ii u_{0} ( - q_{23} + q_{13}\veps^2)\veps^{-2} \notag \\
      {} & {} + u_{12} q_{23}\veps^{-1}  +   q_{13}\veps^{-1},\\
   \tfrac{\dd}{\dd x} q_{21} = {}& - q_{11}\veps +  q_{22}\veps - 1  - q_{21}\veps^{-1},\\
   \tfrac{\dd}{\dd x} q_{22} = {}&  - u_{2}q_{23} + \ii u_{1}q_{23} - u_{0}^2q_{21}\veps^{-1}  + \ii u_{0} q_{21}\veps^{-2} 
      - u_{12} q_{21}\veps^{-1} - \veps q_{12} + \veps^{-2},\\
   \tfrac{\dd}{\dd x} q_{23} = {}& q_{21}u_{2}\veps^{-1}  + \ii q_{21}u_{1}\veps^{-1}  + \ii u_{0} q_{23} - \veps q_{13},\\
   \tfrac{\dd}{\dd x} q_{31} = {}&  - q_{21}u_{2} + \ii q_{21}u_{1} - \ii u_{0} q_{31} + q_{32}\veps - q_{31}\veps ^{-1},\\
   \tfrac{\dd}{\dd x} q_{32} = {}& q_{11}u_{2} - \ii q_{11}u_{1} - u_{0}^2q_{31}\veps^{-1} 
      + \ii u_{0} ( - q_{32}\veps^2 +  q_{31})\veps^{-2} -  u_{12} q_{31} \veps^{-1}.
  \end{align}
\end{subequations} 
Since every $q_{ij}$ does not depend on $u_0$, $u_1$, $u_2$, and $u_{12}$, it follows that the coefficients of nonzero
powers of $u_0$, $u_1$, $u_2$, and $u_{12}$ in~\eqref{eqQbigsys} are equal to zero. We obtain the system
\begin{align}
  q_{31} &= 0, &
  q_{32} + q_{13}\veps &= 0, &
   - q_{32} + q_{13}\veps &= 0, \notag \\
  q_{23} & = 0 , \notag &
   - q_{11}\veps + q_{22}\veps - 1 &= 0, &
   - q_{32}\veps^2 + q_{31} & = 0 \notag.\\
   q_{11} \veps - q_{22}\veps + 2 &= 0 , &
   - q_{22}\veps + 1 &= 0, &
   q_{11} &= 0 , \label{eqQcontr}
\end{align}
By adding the last equation in the first column, 
$q_{11} \veps - q_{22}\veps + 2 = 0$, to the second equation in the other column,
$- q_{11}\veps + q_{22}\veps - 1 = 0$, we obtain the contradition $1=0$. Therefore, system~\eqref{eqQcontr} is not compatible.
This proves Proposition~\ref{PropNonRemDas}: %
there is no $\gsl(2| 1)$-\/matrix~$Q$ satisfying equations~\eqref{eqDasQ} at~$\veps>0$.

\section{Two descriptions of one elimination procedure: an example}\label{appParamNotes}\noindent%
We analyse the following tautological construction: %
by re\/-\/addressing Sasaki,\footnote{A parameter\/-\/dependent zero\/-\/curvature representation for Burgers' equation was considered %
in~\cite{ChernTenenblat86} in the same
context of pseudospherical surfaces as in Sasaki's paper~\cite{Sasaki79}.
We refer to~\cite{Marvan2002} for the ana\-ly\-sis of removability of the
parameter in that zero\/-\/curvature representation for Burgers' equation~\cite{ChernTenenblat86}.} see~\cite{Sasaki79}, we first track how the scaling
symmetry of KdV equation~\eqref{kdv} acts on its standard matrix Lax pair; 
on the other hand, we reveal how these objects are phrased %
in the language of coverings.

\subsection{The Sasaki construction: elimination of a nonremovable parameter}
Recall that the Korteweg\/--\/de Vries equation is
\begin{equation}\tag{\ref{kdv}}
\cE = \left\{ u_t = - u_{xxx} - 6uu_x \right\}.
\end{equation}
Consider the family of coverings $\tau_\eta \colon \tilde{\cE}_\eta \to \cE$ over it,
\begin{subequations}\label{eqSasakiCov}
\begin{align}
v_x &=  2v\eta - (v^2 + u),\\
v_t &= -8\eta^3 v + 4 \eta^2 (v^2 + u) + 2 \eta (-2 v u + u_x) + 2 v^2 u - 2 v u_x + 2 u^2 + u_{xx};
\end{align}
\end{subequations}
these formulas are obtained from the %
$\gsl_2$-\/valued zero\/-\/curvature representation (see~\cite{Sasaki79}),
\[%
  \alpha_\eta = \begin{pmatrix}
    \eta & u \\
    -1   & -\eta
  \end{pmatrix} \Id x 
  + \begin{pmatrix}
    - (4\eta^3 + 2\eta u + u_x) & - (u_{xx} + 2\eta u_x + 4\eta^2u + 2u^2) \\
    4\eta^2 + 2u                & 4\eta^3 + 2\eta u + u_x
  \end{pmatrix} \Id t.
\]%
Let us recall that the parameter~$\eta$ cannot be removed from the zero\/-\/curvature
representations $\alpha_\eta$ by using gauge transformations.
However, it can be \emph{eliminated} by using a wider class of transformations.
Namely, consider the scaling symmetry of equation~\eqref{kdv},
\[%
 x \mapsto \eta x, \quad t \mapsto  \eta^3 t, \quad u \mapsto  \eta^{-2}u, 
\qquad \eta\in\BBR.
\]%
Using it, one transforms the zero\/-\/curvature representation~$\alpha_\eta$ into
\[%
  \alpha'_\eta = \begin{pmatrix}
     1        & \eta u \\
    -\eta^{-1} & -1
  \end{pmatrix} \Id x 
  + \begin{pmatrix}
    - (4 + 2 u + u_x) & - \eta(u_{xx} + 2u_x + 4u + 2u^2) \\
    \eta^{-1}(4 + 2u) & 4 + 2u + u_x
  \end{pmatrix} \Id t.
\]%
The parameter~$\eta$ in~$\alpha'_\eta$ is removable under the gauge transformation
\[%
  g = \begin{pmatrix}
    \eta^{-1/2} & 0 \\
    0          & \eta^{1/2}
  \end{pmatrix} \in C^\infty(\cEinf, GL_2(\BBC)),
\]%
that is, we have that
$ (\alpha'_\eta)^g = \alpha'_\eta\bigr|_{\eta = 1} = \alpha_\eta\bigr|_{\eta = 1}$.

\subsection{How the elimination works in terms of the structure element}
Let us now address the removability of parameter~$\eta$ in coverings~\eqref{eqSasakiCov} 
in terms of the formalism of Cartan's structural element.

For a vector field
\[
X = a\otimes\frac{\dd}{\dd x}  + b\otimes\frac{\dd}{\dd t} 
+ \omega_\sigma\otimes\frac{\dd}{\dd u_\sigma}
+ \varphi\otimes\frac{\dd}{\dd v}, 
\]
the equation for evolution of Cartan's structural element,
\begin{equation}\tag{\ref{eqFNXshadow}}
\frac{\mathrm{d}}{\mathrm{d}\eta} U_\eta = \fnl X, U_\eta \fnr,
\end{equation}
splits into the system
\begin{subequations}\label{eqUcLong}
\begin{align}
-\frac{\mathrm{d}}{\mathrm{d}\eta} v_x = {} & \tilde{D}_x \varphi -
\varphi\frac{\dd v_x}{\dd v} - \omega_\sigma\frac{\dd v_x}{\dd u_\sigma} +
b\left(\frac{\dd v_x}{\dd u_\sigma}u_{\sigma t} + \frac{\dd v_x}{\dd v}v_t -
  \tilde{D}_xv_t\right) - v_t\frac{\dd b}{x}\notag\\
{}&{}+ a\left( - \tilde{D}_xv_x + \frac{\dd v_x}{\dd u_\sigma}u_{\sigma x} +
  \frac{\dd v_x}{\dd v}v_x\right) - v_x\frac{\dd a}{\dd x},\\
-\frac{\mathrm{d}}{\mathrm{d}\eta} v_t = {} & \tilde{D}_t \varphi -
\varphi\frac{\dd v_t}{\dd v} - \omega_\sigma\frac{\dd v_t}{\dd u_\sigma} +
b\left(\frac{\dd v_t}{\dd u_\sigma}u_{\sigma t} + \frac{\dd v_t}{\dd v}v_t -
  \tilde{D}_tv_t\right) - v_t\frac{\dd b}{\dd t}\notag\\
{}&{}+ a\left( - \tilde{D}_t v_x + \frac{\dd v_t}{\dd u_\sigma}u_{\sigma x} +
  \frac{\dd v_t}{\dd v}v_x\right) - v_x\frac{\dd a}{\dd t},\\
\omega_{\sigma x} = {} &  \tilde{D}_x \omega_\sigma 
 - u_{\sigma  t}\frac{\dd b}{\dd x} - u_{\sigma x}\frac{\dd a}{\dd x},\\
\omega_{\sigma t} = {} &  \tilde{D}_t \omega_\sigma 
 - u_{\sigma  t}\frac{\dd b}{\dd t} - u_{\sigma x}\frac{\dd a}{\dd t}.
\end{align}
\end{subequations}
Suppose now that the vector field is vertical:
$X^{\mathrm{v}} =  \omega^{\mathrm{v}}_\sigma\otimes {\dd}/{\dd u_\sigma} + \varphi^{\mathrm{v}}\otimes {\dd}/{\dd v}$.
This simplifies system~\eqref%
{eqUcLong}; it then becomes
\begin{subequations}\label{eqUcVert}
\begin{align}
-\frac{\mathrm{d}}{\mathrm{d}\eta} v_x = {} & \tilde{D}_x \varphi^{\mathrm{v}} -
\varphi^{\mathrm{v}}\frac{\dd v_x}{\dd v} - \omega^{\mathrm{v}}_\sigma\frac{\dd v_x}{\dd u_\sigma}, \\
-\frac{\mathrm{d}}{\mathrm{d}\eta} v_t = {} & \tilde{D}_t \varphi^{\mathrm{v}} -
\varphi^{\mathrm{v}}\frac{\dd v_t}{\dd v} - \omega^{\mathrm{v}}_\sigma\frac{\dd v_t}{\dd u_\sigma}, \\ 
\omega^{\mathrm{v}}_{\sigma x} = {} &  \tilde{D}_x \omega^{\mathrm{v}}_\sigma ,\\
\omega^{\mathrm{v}}_{\sigma t} = {} &  \tilde{D}_t \omega^{\mathrm{v}}_\sigma .
\end{align}
\end{subequations}
Let us use the Ansatz
\[
\omega^{\mathrm{v}} = \omega - au_x - bu_t, \quad \varphi^{\mathrm{v}} = \varphi - av_x -bu_t,
\]
assuming that 
$a=a(x, t, \eta)$,\ $b=b(x, t, \eta)$,\ $\varphi = \varphi(\eta,
u, v)$, and~$\omega = \omega(\eta, u, v, u_x, u_{xx})$. 
By construction, the unknowns $\omega^{\mathrm{v}}$ and $\varphi^{\mathrm{v}}$
satisfy system~\eqref{eqUcVert}. 
Using the analytic software \texttt{Jets}~\cite{MarvanJets} and
\texttt{Crack}~\cite{SsTools}, we find the solution 
\begin{align*}
a ={} & 24c_4t\eta^3 + 2c_4x\eta + \tfrac{1}{\eta}(c_6+x),\\
b ={} & 6c_4t\eta + \tfrac{1}{\eta}(-c_7 + 3t),\\
\omega = {} & 4c_4\eta^3 - 4c_4u\eta + u_xc_4 + \tfrac{1}{\eta}(-\tfrac{1}{2}u_xc_3 - 2u) + \tfrac{1}{2\eta^2}u_x,\\
\varphi ={} & 2c_4\eta^2 - c_4v^2 - c_3v - c_4u + \tfrac{c_3}{2\eta}(v^2 + u) - \tfrac{1}{2\eta^2} (v^2 + u),
\end{align*}
which contains four arbitrary constants $c_3$, $c_4$, $c_6$, and~$c_7$.

Let us set $c_3=0$, $c_4=-1/(2\eta^2)$ at $\eta\neq0$, $c_6=0$, and $c_7=0$.  This determines the solution
which corresponds to the lift of Galilean symmetry of~\eqref{kdv}:
\[%
  X_2 = -2\eta( 6t\,{\dd}/{\dd x}
    + \,{\dd}/{\dd u}
    + \dots)
    - {\dd}/{\dd v}. 
\]%
On the other hand, set $c_3=1/\eta$ if~$\eta\neq0$ and let $c_4=0$, $c_6=0$, and $c_7=0$.  This yields the
solution which corresponds to the lift of scaling symmetry of~\eqref{kdv}; namely, we have that
\begin{equation}\label{eqXScal}
  X_1 = \eta^{-2}(   - x\,{\partial}/{\partial x}
    - 3t\,{\partial}/{\partial t}
    + 2u  \,{\partial}/{\partial u}
    + \ldots
    + v  \,{\partial}/{\partial v} ). 
\end{equation}
The integral curves of vector field~\eqref{eqXScal} encode the transformation
\begin{equation}\label{eqSasakiScaling}
  x \mapsto \eta x, \quad
  t \mapsto \eta^3 t, \quad
  u \mapsto \eta^{-2} u, \quad
  v \mapsto \eta^{-1} v.
\end{equation}
Its action on the covering~$\tau_\eta$ in~\eqref{eqSasakiCov} results in the covering
$\tau' = \tau_\eta \bigr|_{\eta=1}$, which is described by the formulas
\begin{align*}
v_x &= 2v - (v^2 + u),\\
v_t &= -8 v + 4 v^2 + 4u -4 v u + 2u_x + 2 v^2 u - 2 v u_x + 2 u^2 + u_{xx}.
\end{align*}
It is readily seen 
that the covering~$\tau'$ is the image of zero\/-\/curvature 
representation~$(\alpha'_\eta)^g$ under a swapping of representations for the Lie algebra 
at hand. This is shown in the following diagram:
\begin{equation}\label{CDSasaki}
  \begin{CD}
    \alpha_\eta
    @>{\text{scaling}}>>
    \alpha'_\eta
    @>g>>
    \alpha'_\eta \bigr|_{\eta=1} \\
    @|        @.           @VV{\nabla}V \\
    \alpha_\eta
    @>{\nabla}>>
    \tau
    @>{\eqref{eqSasakiScaling}}>>
    \tau'.
  \end{CD}
\end{equation}
We conclude that the problem of finding transformations (which are possibly not gauge)
that eliminate the parameter in a given family of zero\/-\/curvature representations
can be approached via a solution of equation~\eqref{eqFNXshadow} in the family of coverings
which are the $(\rho\rightleftarrows\boldsymbol{\varrho})$-\/avatars of those zero\/-\/curvature 
representations.

\subsection{Overview: taxonomy of the parameters}
Depending on their elimination scenario, ``removable'' parameters in zero\/-\/curvature 
representations are classified as follows:
\begin{enumerate}
\item %
There are parameters which are truly removable under the action of smooth
families of gauge transformations
(see~\cite{Marvan2002,Marvan2010} by Marvan and~\cite{Sakovich95,Sakovich2004} by Sakovich).
\item There could be zero\/-\/curvature representations~$\alpha_\lambda$ which are (piecewise-)\/smooth in the parameter~$\lambda\in\mathcal{I}\subseteq\BBC$ but such that the families~$S_\lambda$ of gauge transformations removing the parameter are \emph{not} smooth at all points~$\lambda\in\mathcal{J}\subseteq\mathcal{I}$, where the set~$\mathcal{J}$ is
\begin{enumerate}
\item finite,
\item countable,
\item everywhere dense in~$%
\mathcal{I}$ but not amounting to it, or
\item equal to the entire set~$\mathcal{I}$ of admissible values of the parameter~$\lambda$.
\end{enumerate}
This analytic curiosity would be the threshold limit of the preceding case.
\item Next, there are parameters which cannot be removed by using gauge transformations 
but which indicate the presence of conserved currents in zero\/-\/curvature representations
and the reducibility of such representations,\footnote{For example, consider
a ``fake'' $\gsl_2$-\/valued zero\/-\/curvature representation
$\alpha
    = \begin{pmatrix}
      0 & X_1 + \lambda X_2 \\
      0 & 0
    \end{pmatrix}\Id x 
    + \begin{pmatrix}
      0 & T_1 + \lambda T_2 \\
      0 & 0
    \end{pmatrix} \Id t$ 
for an equation $\cE$ possessing two conserved currents
$\bar{D}_t X_i = \bar{D}_x T_i$, here~$i=1,2$.} %
(see~\cite{MarvanRed} and~\cite[\S\,12]{GDE2012}).
\item %
There are parameters which vanish under the action of those
symmetries of the underlying differential equation which cannot be lifted to
the covering Maurer\/--\/Cartan equation (see~\cite{Sasaki79,LeviSymTu}).
\item Finally, there are parameters which can be eliminated by the same procedure as 
in the %
preceding case but by using \emph{shadows} of nonlocal symmetries in some 
auxiliary covering over the equation at hand (namely, \emph{not} in the covering 
which grasps the ZCR geometry but in an extension of the equation's geometry by 
a set of ``nonlocalities''),
see~\cite{Cieslinski93Nonlocal,Cieslinski93Group,Cieslinski94}.
\end{enumerate}

\subsubsection*{Acknowledgements}
The authors are grateful to 
I.~S.~Krasil'shchik,
D.~A.~Leites,
M.~Marvan,
M.~A.~Nesterenko, 
P.~J.~Olver,
W.~M.~Seiler, and 
A.~M.~Verbovetsky for helpful correspondence and constructive criticisms. 
The authors thank P.~Mathieu for his attention to this work;
the authors are grateful to the anonymous referees for remarks %
and advice.

This research was done in part while the first author was visiting at the MPIM (Bonn) and the second author
was visiting at Utrecht University and New York University Abu Dhabi; the hospitality and support of these institutions are gratefully
acknowledged. The research of the first author was partially supported by JBI~RUG project~106552 (Groningen);
the second author was supported by ISPU scholarship for young scientists and WCMCS post-doctoral fellowship.

\end{document}